\documentclass[a4paper]{article}

\usepackage[margin=2.9cm]{geometry}
\usepackage[utf8]{inputenc}
\usepackage{amsmath}
\usepackage{amssymb}
\usepackage{amsthm}
\usepackage{mathrsfs}
\usepackage{dsfont}
\usepackage{upgreek}
\usepackage{graphicx}
\usepackage[pdftex,dvipsnames]{xcolor}
\usepackage{enumitem}
\usepackage[round,comma,authoryear]{natbib}
\usepackage{subcaption}

\graphicspath{{/home/uli/PAPERS/lassodist/plots/}}

\theoremstyle{plain}
\newtheorem{theorem}{Theorem}
\newtheorem{lemma}[theorem]{Lemma}
\newtheorem{proposition}[theorem]{Proposition}
\newtheorem{corollary}[theorem]{Corollary}

\newtheorem*{example*}{Example}
\newtheorem{example}{Example}
\newtheorem{remark}{Remark}
\newtheorem*{remark*}{Remark}
\newtheorem{acknowledgement*}{Acknowledgement}

\newcommand{\rk}{\operatorname{rk}}
\newcommand{\sign}{\operatorname{sgn}}

\newcommand{\col}{\operatorname{col}}
\newcommand{\E}{\operatorname{\mathbb{E}}}
\newcommand{\argmin}{\operatorname*{\arg\min}}

\newcommand{\sgn}{\operatorname{sgn}}

\newcommand{\Bzero}{\mathcal{B}_0}
\newcommand{\BM}{\mathcal{B}_\mathcal{M}}
\newcommand{\M}{\mathcal{M}}
\newcommand{\Mstruct}{\mathscr{M}}
\newcommand{\B}{\mathscr{B}}
\renewcommand{\O}{\mathcal{O}}

\newcommand{\R}{\mathbb{R}}

\newcommand{\Rge}{\R_{>0}}
\newcommand{\Rgeq}{\R_{\geq 0}}
\newcommand{\ind}{\mathds{1}}
\newcommand{\eps}{\varepsilon}

\newcommand{\betaL}{\hat\beta_{\textnormal{\tiny L}}}
\newcommand{\betaLj}{\hat\beta_{\textnormal{\tiny L},j}}
\newcommand{\betaLx}[1]{\hat\beta_{\textnormal{\tiny L},{#1}}}
\newcommand{\betaLS}{\hat\beta_{\textnormal{\tiny LS}}}

\newcommand{\iid}{\overset{\text{\tiny{iid}}}{\sim}}

\renewcommand{\emptyset}{\varnothing}

\newcommand{\pd}[2]{\frac{\partial #1}{\partial #2}}
\newcommand\myatop[2]{\genfrac{}{}{0pt}{}{#1}{#2}}

\begin{document}

\title{On the Distribution, Model Selection Properties and Uniqueness
of the Lasso Estimator in Low and High Dimensions}

\author{Karl Ewald and Ulrike Schneider\\[2ex]
Vienna University of Technology} 

\date{}

\maketitle

\begin{abstract}
We derive expressions for the finite-sample distribution of the Lasso
estimator in the context of a linear regression model in low as well
as in high dimensions by exploiting the structure of the optimization
problem defining the estimator. In low dimensions, we assume full rank
of the regressor matrix and present expressions for the cumulative
distribution function as well as the densities of the absolutely
continuous parts of the estimator. Our results are presented for the
case of normally distributed errors, but do not hinge on this
assumption and can easily be generalized. Additionally, we establish
an explicit formula for the correspondence between the Lasso and the
least-squares estimator. We derive analogous results for the
distribution in less explicit form in high dimensions where we make no
assumptions on the regressor matrix at all. In this setting, we also
investigate the model selection properties of the Lasso and show that
possibly only a subset of models might be selected by the estimator,
completely independently of the observed response vector. Finally, we
present a condition for uniqueness of the estimator that is necessary
as well as sufficient.

\medskip


\end{abstract}

\section{Introduction} 
\label{sec:intro}

The distribution of the Lasso estimator \citep{Tibshirani96} has been
an object of study in the statistics literature for a number of years.
The often cited paper by \cite{KnightFu00} gives the asymptotic
distribution of the Lasso in the framework of conservative model
selection in a low-dimensional (fixed-$p$) framework by listing the
limit of the corresponding stochastic optimization.
\cite{PoetscherLeeb09} derive explicit expressions of the distribution
in finite samples as well as asymptotically for all large-sample
regimes of the tuning parameter (``conservative'' as well as
``consistent model selection'') in the framework of orthogonal
regressors. More recently, \cite{Zhou14} gives high-level information
on the finite-sample distribution for arbitrary designs in low and
high dimensions, geared towards setting up a Monte-Carlo approach to
infer about the distribution. In \cite{EwaldSchneider18}, the
large-sample distribution of the Lasso is derived in a low-dimensional
framework for the large-sample regime of the tuning parameter not
considered in \cite{KnightFu00}. Moreover, \cite{JagannathUpadhye18}
consider the characteristic function of the Lasso to obtain
approximate expressions for the marginal distribution of
one-dimensional components of the Lasso when these components are
``large'', therefore not having to consider the atomic part of the
estimator.

In this paper, we exactly and completely characterize the distribution
of the Lasso estimator in finite samples in the context of a linear
regression model with normal errors. In low dimensions, we give
formulae for the cumulative distribution function (cdf), as well as
the density functions conditional on which components of the estimator
are non-zero. We do so assuming full column rank of the regressor
matrix. Our results do not hinge on the normality assumption of the
errors, but can easily be extended to more general error
distributions. We also exactly quantify the correspondence between the
Lasso and least-squares (LS) estimator, depending on the regressor
matrix and tuning parameters only.

In a high-dimensional setting, we make absolutely no assumptions on
the regressor matrix. We give formulae for the probability of the
Lasso estimator falling into a given set and exactly quantify the
relationship between the Lasso estimator and the data object $X'y$.
Through this relationship, we also learn that the Lasso may never
select certain models, this property depending only on the regressor
matrix and the penalization weights and being independent of the
observed response vector. In fact, we can characterize a so-called
structural set that contains all covariates that are part of a Lasso
model model for some response vector. This structural set can be
identified by how the row space of the regressor matrix intersects a
cube centered at the origin whose side lengths are determined by the
penalization weights. The set may not contain all indices, in which
case the Lasso estimator will rule out certain covariates for all
possible observations of the dependent variable. This is related to
the idea of so-called SAFE rules \citep{TibshiraniEtAl12} that can
discard covariates for Lasso solutions for a fixed value of the
dependent variable.

Finally, we present a condition for uniqueness of the Lasso estimator
that is both necessary and sufficient, again related to how the row
space of the regressor matrix intersects the above mentioned cube.
Previously, only a sufficient condition for uniqueness has been known
\citep[see e.g.][]{Tibshirani13,AliTibshirani19}. The results
quantifying the relationship between the Lasso and the LS estimator or
$X'y$, respectively, are in fact completely independent of the error
distribution and merely utilize the given values of the dependent
variable and the regressor matrix. The results on model selection
properties and uniqueness use the regressor matrix and penalization
weights only.

The paper is organized as follows. We introduce the setting and
notation in Section~\ref{sec:setting} and state basic results used
throughout the paper. The low-dimension case is treated in
Section~\ref{sec:LOW}, whereas we consider the high-dimensional case
in Section~\ref{sec:HIGH}. We conclude in
Section~\ref{sec:conclusion}.

\section{Setting, Notation, and Basic Results} 
\label{sec:setting}
Consider the linear model
\begin{equation} 
\label{eqn:model}
y = X \beta + \eps,
\end{equation}
where $y$ is the observed $n \times 1$ data vector, X is the $n \times
p$ regressor matrix which is assumed to be non-stochastic, $\beta \in
\R^p$ is the true parameter vector and $\eps$ the unobserved error
term. We assume that $X'\eps$ follows a
$N(0,\sigma^2X'X)$-distribution with $\sigma^2 > 0$, which, for
instance, is the case when the components of $\eps$ are independent
and identically distributed according to a
$N(0,\sigma^2)$-distribution. Our results depend on the distribution
of $X'\eps$ only, and we chose the normal distribution for
presentation purposes.\footnote{In fact, some of the following results
are completely independent of the error distribution. These are
Lemma~\ref{lem:kkt} and Corollary~\ref{cor:kkt},
Theorems~\ref{thm:shrink} and \ref{thm:shrinkHIGH},
Corollary~\ref{cor:modelselection} as well as
Theorems~\ref{thm:struct} and \ref{thm:unique}.}
 
We consider the \emph{weighted Lasso estimator} $\betaL$, defined as a
solution to the minimization problem
\begin{equation} 
\label{eqn:lassodef}
\min_{\upbeta \in \R^p} L(\upbeta) = \min_{\upbeta \in \R^p}  
\|y - X\upbeta\|^2 + 2\sum_{j=1}^p \lambda_{n,j} |\upbeta_j|,
\end{equation}
where $\lambda_{n,j}$, are non-negative user-specified tuning
parameters that will typically depend on $n$. To ease notation, we
shall suppress this dependence for the most part and write
$\lambda_{n,j} = \lambda_j$ for each $j$. Note that if $\lambda_j = 0$
for all $j$, the weighted Lasso is equal to LS estimation, and that
$\lambda_1 = \dots = \lambda_p > 0$ corresponds to the classical Lasso
estimator as proposed by \cite{Tibshirani96}, to which case we also
refer to by uniform tuning. For later use, let $\lambda =
(\lambda_{1},\dots,\lambda_{p})'$ and define $\M_0 = \{j : \lambda_j =
0\}$, the index set of all unpenalized coefficients. If $\M_0 \neq
\emptyset$, we speak of partial tuning. Note that $\M_0$ contains the
indices of covariates that will be part of any model chosen by the
Lasso. We stress dependence on the unknown parameter $\beta$ when it
occurs, but do not specify dependence on $X$, $y$ or $\lambda$ as
these quantities are available to the user.

The following notation will be used throughout the paper. Let
$\phi_{(\mu, \Sigma)}$ denote the Lebesgue-density of a normally
distributed random variable with mean $\mu$ and covariance matrix
$\Sigma$, and let $\Phi$ be the cdf of a univariate standard normal
distribution. For a vector $m \in \R^p$ and an index set $I \subseteq
\{1,\dots,p\}$, the vector $m_I \in \R^{|I|}$ contains only the
components of $m$ corresponding to the elements of $I$. We write $|I|$
for the cardinality of $I$, and $I^c$ for $\{1,\dots,p\}\setminus I$,
the complement of $I$. The $1$-norm of $m$ is denoted by $\|m\|_1$
whereas the $2$-norm is simply denoted by $\|m\|$. For $x \in \R$, let
$\sign(x) = \ind_{\{x>0\}} - \ind_{\{x < 0\}}$ where $\ind$ is the
indicator function. For a set $A \subseteq \R^p$, the set $m + A= A +
m$ is defined as $\{m + z: z \in A\}$, with a analogous definitions
for $A - m$ and $m - A$. We denote the Cartesian product by $\prod$
and the kernel, column space, and rank of a matrix $C$ by $\ker(X)$,
$\col(C)$, and $\rk(C)$, respectively. The columns of $C$ are denoted
by $C_j$ whereas $C_I$, for some index set $I$, is the matrix
containing the $|I|$ columns of $C$ corresponding to indices in $I$
only. For a quadratic matrix $C$, $|C|$ denotes the determinant of
$C$. We use $\Rge$ for the positive, and $\Rgeq$ for the non-negative
real numbers.

Let $\{D_-,D_0,D_+\}$ be a partition of $\{1,\dots,p\}$ into three
sets, some of which may be empty. It will be convenient to also
describe this partition by a vector $d \in \{-1,0,1\}^p$ with $d_j =
\ind_{\{j \in D_+\}} - \ind_{\{j \in D_-\}}$. For such $d$, we denote
by $\O^d = \{z \in \R^p : \sign(z_j) = d_j \text{ for } j=1,\dots,p\}
= \{z \in \R^p : z_j < 0 \text{ for } j \in D_-, z_j = 0 \text{ for }
j \in D_0, z_j > 0 \text{ for } j \in D_+\}$. Note that $m + \beta \in
\O^d$ is short-hand notation for $m_j < -\beta_j$ for $j \in D_-$,
$m_j = -\beta_j$ for $j \in D_0$ and $m_j > -\beta_j$ for $j \in D_+$.
We write $D_\pm$ for $D_-\cup D_+$.

Finally, we state the conditions that characterize solutions to
\eqref{eqn:lassodef}, known as the Kuhn-Karush-Tucker (KKT) conditions
for the Lasso \citep[see e.g.][with the slight adaptation that we use
componentwise tuning]{Tibshirani13} which form the basis of most
proofs and will be used throughout the article.

\begin{lemma} \label{lem:kkt} 
We have that
$$
b \in \argmin_{\upbeta \in \R^p} L(\upbeta) \iff X'y \in X'Xb + \prod_{j=1}^p B_j(b_j),
$$
where 
$$
B_j(b_j) = \begin{cases}
\{\sign(b_j)\lambda_j\} & b_j \neq 0 \\
[-\lambda_j,\lambda_j] & b_j = 0.
\end{cases}
$$
\end{lemma}
Plugging in $y = X\beta + u$, we rewrite this as
\begin{corollary} \label{cor:kkt}
$$
b \in \argmin_{\upbeta \in \R^p} L(\upbeta) \iff X'\eps \in A_\beta(b)
$$
with $A_\beta(b) = X'X(b-\beta) + \prod_{j=1}^p B_j(b_j)$.
\end{corollary}

\section{The Low-dimensional Case} 
\label{sec:LOW}

Throughout this section, \emph{we assume that $X$ has full column rank
$p$}, implying that we are considering the low-dimensional setting
where $p \leq n$. This assumption is used in the following through the
fact that $W = X'\eps$ follows a non-degenerate normal distribution on
$\R^p$ in the distributional statements on $\betaL$
(Theorem~\ref{thm:prob}, Corollary~\ref{cor:allzero}). It is
additionally used through the fact that the true parameter is properly
identified in the distributional statements concerning the estimation
error $\betaL - \beta$ (Corollary~\ref{cor:prob},
Proposition~\ref{prop:density}, Theorem~\ref{thm:cdf}). We also rely
on this assumption in Theorem~\ref{thm:shrink} through the
invertibility of $X'X$ and the existence of the LS estimator.

\begin{theorem} \label{thm:prob}
Let $z \in \R^p$ and let $d = \sgn(z)$ with $\{D_-,D_0,D_+\}$ being the
corresponding partition of $\{1,\dots,p\}$.
\begin{align*}
& P(\betaLj \leq z_j \text{ for } j \in  D_-, \; \betaLj = 0 \text{ for } j \in D_0, \; 
\betaLj \geq z_j \text{ for } j \in D_+) \\[2ex]
& = \mathop{\int\dots\int}_{\myatop{m_j \geq z_j - \beta_j}{j \in D_+}} \!
\mathop{\int\dots\int}_{\myatop{s_j \in [-\lambda_j,\lambda_j]}{j \in D_0}} \!
\mathop{\int\dots\int}_{\myatop{m_j \leq z_j - \beta_j}{j \in D_-}} \!
\phi_{(0, \sigma^2X'X)}(X'Xm_\beta + s_\lambda) |X_{D_\pm}'X_{D_\pm}|\, dm_{D_-}ds_{D_0}dm_{D_+},
\end{align*}
where $m_\beta$ and $s_\lambda \in \R^p$ are given by
$(m_\beta)_{D_\pm} = m_{D_\pm}$, $(m_\beta)_{D_0} =
-\beta_{D_0}$ and $(s_\lambda)_{D_-} = -\lambda_{D_-}, s_{D_+} =
\lambda_{D_+}$, $(s_\lambda)_{D_0} = s_{D_0}$, respectively.
\end{theorem}

\begin{proof} 
We need to compute the probability of the event $\{\betaL \in B_z\}$
where $B_z = \{b \in \R^p : b_j \leq z_j \text{ for } j \in D_-, \, b_j =
0 \text{ for } j \in D_0, \, b_j \geq z_j \text{ for } j \in D_+\}$. By
Corollary~\ref{cor:kkt}, this event is equivalent to the event $\{W
\in A_\beta(B_z)\}$ with $A_\beta(B_z) = \cup_{b \in B_z} A_\beta(b)$
and $A_\beta(b)$ as defined in Corollary~\ref{cor:kkt}. As $W = X'\eps
\sim N(0,\sigma^2 X'X)$, the probability we are looking for is
therefore given by
$$
\int_{A_\beta(B_z)} \phi_{(0,\sigma^2X'X)}(w)dw.
$$
We now look at the structure of the set $A_\beta(B_z)$ more
concretely. Since $\sgn(b) = \sgn(z)$ for all $b \in B_z$, we get
\begin{align*}
A_\beta(B_z) & \; = \;\; X'XB_z - X'X\beta + \prod_{j=1}^p B_j(z_j) \\ 
= \;\; & \{X'X(b-\beta) : b_j \leq z_j \text{ for } j \in D_-, \,
b_j = 0 \text{ for } j \in D_0, \, b_j \geq z_j \text{ for } j \in D_+\} \\
& + \{s : s_j = -\lambda_j \text{ for } j \in
D_-, \, |s_j| \leq \lambda_j \text{ for } j \in D_0, \, s_j = \lambda_j \text{ for } j \in D_+\},
\end{align*}
which, after applying the substitution $w = X'Xm_\beta + s_\lambda$, 
yields the claim.
\end{proof}

Theorem~\ref{thm:prob} gives the distribution of $\betaL$. The
dependence on the unknown parameter $\beta$ arises in the shift
$(m_\beta)_{D_0} = -\beta_{D_0}$ as well as in the limits for the
variables of integration $m_{D_+}$ and $m_{D_-}$. In case the
regressors are orthogonal, more concretely, if $X'X = I_p$, the
probability expression in Theorem~\ref{thm:prob} can be written as
\begin{align*}
\prod_{j=1}^p \bigg( & 
\ind_{\{j \in D_-\}} \int_{-\infty}^{z_j - \beta_j} \phi_{(0,\sigma^2)}(m-\lambda_j)dm 
+ \ind_{\{j \in D_0\}} \int_{-\lambda_j}^{\lambda_j} \phi_{(0,\sigma^2)}(s - \beta_j)ds \\
& + \ind_{\{j \in D_+\}} \int_{z_j -\beta_j}^\infty \phi_{(0,\sigma^2)}(m+\lambda_j)dm\bigg) \\
= \prod_{j=1}^p \bigg( & \ind_{\{j \in D_-\}} P(Z_j + \beta_j \leq z_j - \lambda_j) 
+ \ind_{\{j \in D_0\}} P(-\lambda_j \leq Z_j + \beta_j \leq \lambda_j) \\ 
& + \ind_{\{j \in D_+\}} P(Z_j + \beta_j \geq z_j + \lambda_j) \bigg),
\end{align*}
where $Z_j \iid N(0,\sigma^2)$, which is consistent with the
well-known fact that the Lasso is equivalent to componentwise
soft-thresholding in this case, also treated in
\cite{PoetscherSchneider09}.

The distribution of the estimation error $\hat u = \betaL - \beta$ can
now be derived from Theorem~\ref{thm:prob} as a corollary.

\begin{corollary} \label{cor:prob} 
Let $z \in \R^p$. Let $d = \sgn(z + \beta) \in \{-1,0,1\}^p$ and let
$\{D_-,D_0,D_+\}$ be the corresponding partition of $\{1,\dots,p\}$. Then
\begin{align*}
& P(\hat u_j \leq z_j \text{ for } j \in  D_-, \; 
\hat u_j = z_j \text{ for } j \in D_0, \; \hat u_j \geq z_j
\text{ for } j \in D_+) \\[2ex]
& = \mathop{\int\dots\int}_{\myatop{m_j \geq z_j}{j \in D_+}} \!
\mathop{\int\dots\int}_{\myatop{s_j \in [-\lambda_j,\lambda_j]}{j \in D_0}} \!
\mathop{\int\dots\int}_{\myatop{m_j \leq z_j}{j \in D_-}} \!
\phi_{(0, \sigma^2X'X)}(X'Xm_\beta + s_\lambda) |X_{D_\pm}'X_{D_\pm}| \,dm_{D_-}ds_{D_0}dm_{D_+},
\end{align*}
where $m_\beta$ and $s_\lambda \in \R^p$ are given by
$(m_\beta)_{D_\pm} = m_{D_\pm}$, $(m_\beta)_{D_0} = -\beta_{D_0}$ and
$(s_\lambda)_{D_-} = -\lambda_{D_-}, s_{D_+} = \lambda_{D_+}$,
$(s_\lambda)_{D_0} = s_{D_0}$, respectively.
\end{corollary}

\begin{proof}
Apply Theorem~\ref{thm:prob} using $z + \beta$ rather than $z$.
\end{proof}

Another direct consequence of Theorem~\ref{thm:prob} is a concrete
formula for the probability of the extreme event of the Lasso setting
all components equal to zero.

\begin{corollary} \label{cor:allzero}
$$
P(\betaL = 0) = \int_{-\lambda_p}^{\lambda_p} \dots \int_{-\lambda_1}^{\lambda_1}
\phi_{(X'X\beta,\sigma^2 X'X)}(w) \, dw.
$$
\end{corollary}

\begin{proof}
Apply Theorem~\ref{thm:prob} using $z =0$.
\end{proof}

\begin{remark} \label{rem:sets}
To illustrate the structure behind the proof of
Theorem~\ref{thm:prob}, note that the equivalence of the events
$\{\hat\beta \in B_z\}$ and $\{W \in A_\beta(B_z)\}$ is shown through
through Corollary~\ref{cor:kkt}. The equivalence holds due to the
structure of the optimization problem defining $\betaL$ and does not
depend on the distribution of $W = X'\eps$. In this sense, the
distributional results do not hinge on the normality assumption of the
errors and can easily be generalized to other error distributions. The
relationship and shape of the sets $B_z$ and $A_\beta(B_z)$ is
illustrated in Figure~\ref{fig:intAreas}. Note that $A_\beta$ depends
on $\lambda$, whereas $B_z$ does not.
\end{remark}

\begin{figure}[htp]
\centering
\begin{subfigure}{\textwidth}
\includegraphics[width=0.36\textwidth]{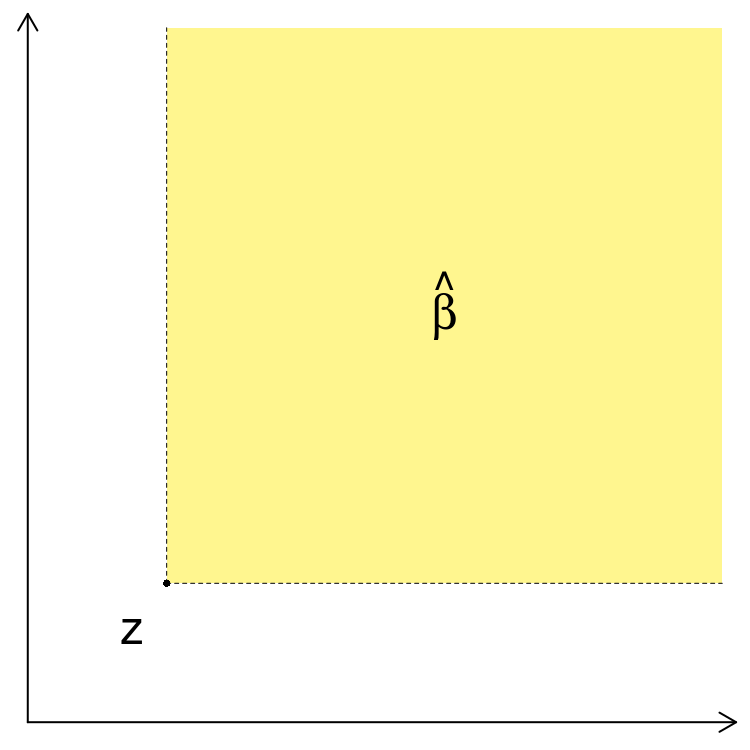} \hfill
\includegraphics[width=0.36\textwidth]{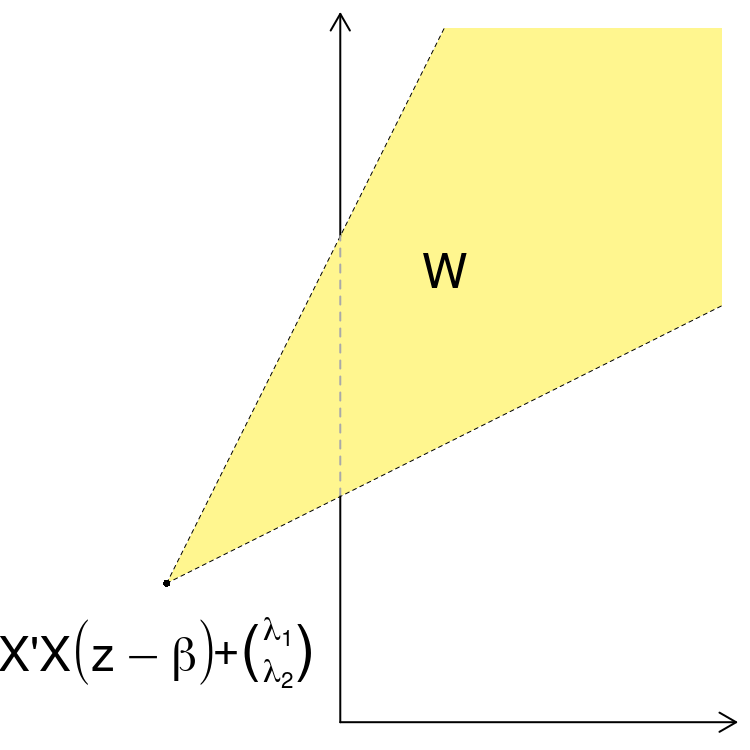}
\caption{$z_1, z_2 > 0$}
\end{subfigure}

\bigskip
\bigskip

\begin{subfigure}{\textwidth}
\includegraphics[width=0.36\textwidth]{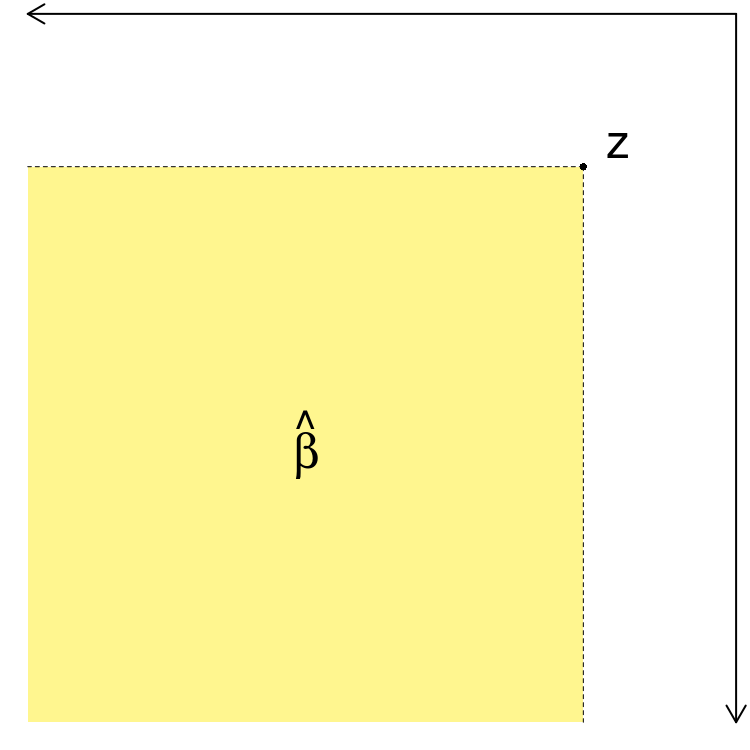} \hfill
\includegraphics[width=0.36\textwidth]{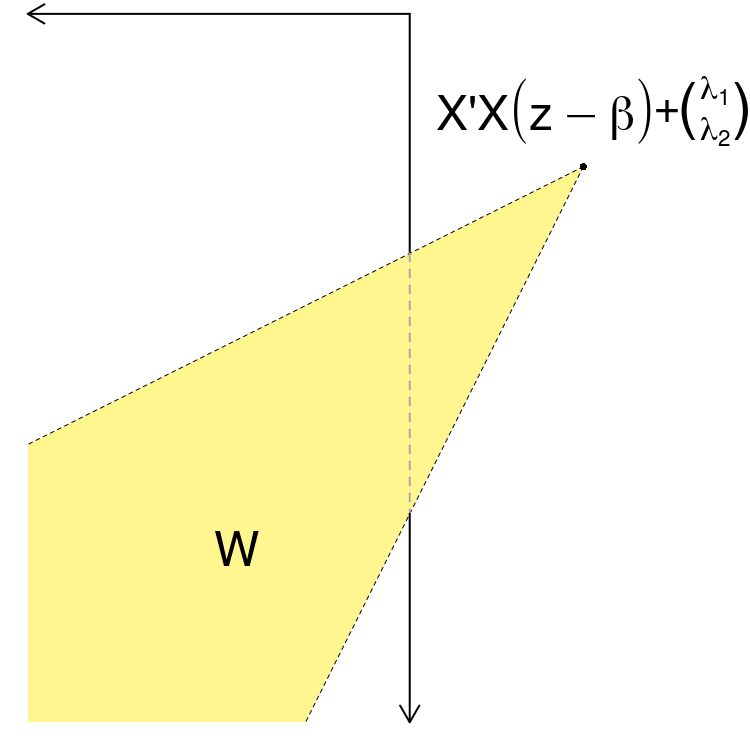}
\caption{$z_1, z_2 < 0$}
\end{subfigure}

\bigskip
\bigskip

\begin{subfigure}{\textwidth}
\includegraphics[width=0.18\textwidth]{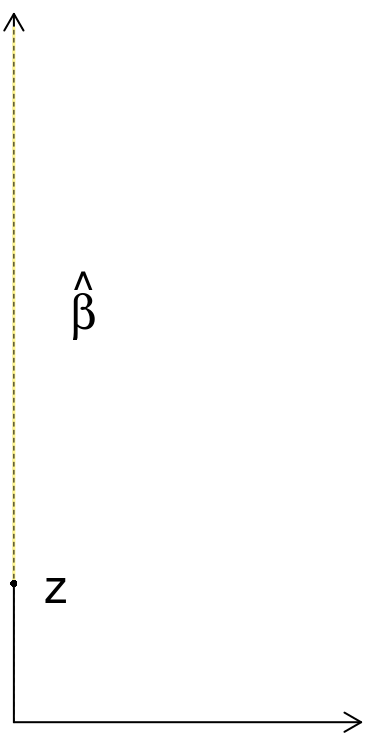} \hfill
\includegraphics[width=0.36\textwidth]{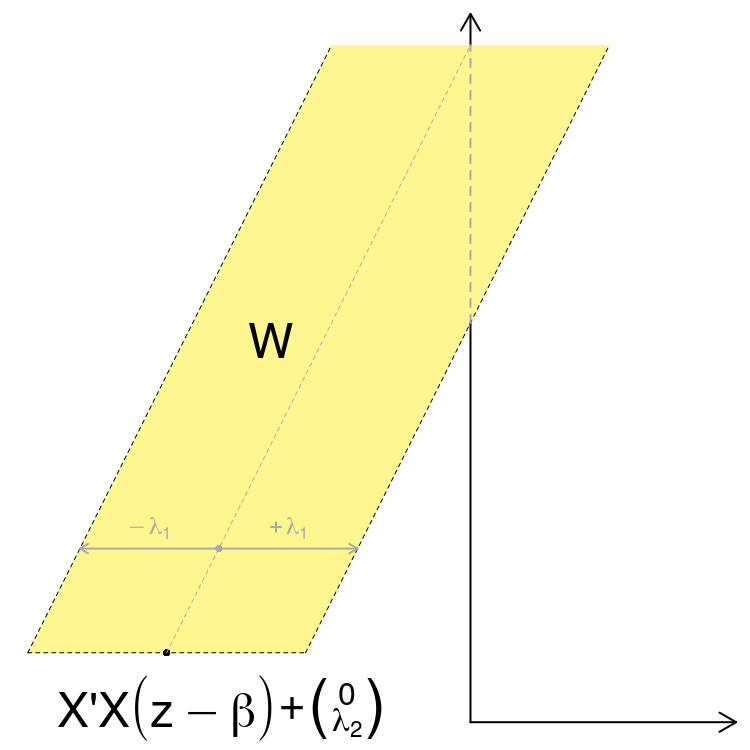}
\caption{$z_1 = 0$ and $z_2 > 0$}
\end{subfigure}

\caption{\label{fig:intAreas} The sets $B_z$ are displayed on the
left-hand side, the corresponding sets $A_\beta(B_z)$ are displayed on
the right-hand side. Illustration for $p=2$ and various values of $z$,
see Remark~\ref{rem:sets} for details.}
\end{figure}

\begin{remark} \label{rem:normLOW}
Theorem~\ref{thm:prob}, Corollary~\ref{cor:prob},
Corollary~\ref{cor:allzero}, Proposition~\ref{prop:density} and
Theorem~\ref{thm:cdf} do not rely on the normal distribution, as just
mentioned. Indeed, the results equally hold for any other absolutely
continuous distribution of $X'\eps$ (with respect to Lebesgue
measure), only the expression $\phi_{(0,\sigma^2X'X)}$ would have to
be replaced by the corresponding density function of $X'\eps$.
Moreover, the results also hold for discrete $X'\eps$ in which case
the integral would have to be replaced by a sum, and the density
function by the corresponding probability mass function.
\end{remark}

Theorem~\ref{thm:prob} now puts us into a position to fully specify
the distribution of the Lasso estimator. In case $\lambda_j>0$ for all
$j$, one easily sees from the preceding corollary that this
distribution is not absolutely continuous with respect to the
$p$-dimensional Lebesgue-measure, and thus no density exists. One can,
however, represent the distribution through Lebesgue-densities after
conditioning on which components of the estimator are negative, equal
to zero, and positive, which we shall do in the sequel.

\begin{proposition} \label{prop:density}
The distribution of $\hat u = \betaL - \beta$, conditional on the
event $\{\betaL \in \O^d\}$, can be represented by a
$\|d\|_1$-dimensional Lebesgue-density given by
\begin{align*}
f^d(z_{D_\pm}) = \frac{\ind\{z_\beta + \beta \in \O^d\}}{P(\betaL \in \O^d)} \!\!
\mathop{\int \dots \int}_{\myatop{s_j \in [-\lambda_j,\lambda_j]}{j \in D_0}}
\phi_{(0,\sigma^2 X'X)} \left(X'Xz_\beta + s_\lambda\right) |X'_{D_\pm}X_{D_\pm}| \, ds_{D_0},
\end{align*}
where $z_\beta$ is defined by $(z_\beta)_{D_\pm} = z_{D_\pm}$, and
$(z_\beta)_{D_0} = -\beta_{D_0}$ and $s_\lambda$ is defined by
$(s_\lambda)_{D_-} = -\lambda_{D_-}$, $(s_\lambda)_{D_+} =
\lambda_{D_+}$, and $(s_\lambda)_{D_0} = s_{D_0}$. Note that the
constants $P(\betaL \in \O^d)$ can be calculated using
Corollary~\ref{thm:prob}.
\end{proposition}

\begin{proof}
Observe that
$$
f^d(z_{D_\pm}) = \left(\pd{}{z_j}\right)_{j \in D_\pm} P\left(\hat u_j
\leq z_j \text{ for } j \in D_\pm | \betaL \in \O^d\right),
$$
and note that by Corollary~\ref{cor:prob}, for any $z \in \R^p$ with
$z + \beta \in \O^d$ we have
\begin{align*}
P&\left(\hat u_j \leq z_j \text{ for } j \in D_-, \hat u_j \geq z_j
\text{ for } j \in D_+| \betaL \in \O^d\right) P(\betaL \in \O^d)\\ 
& = P\left(\hat u_j \leq z_j \text{ for } j \in D_-,
\; \betaLj = 0 \text{ for } j \in D_0, \; \hat u_j \geq z_j \text{ for } j \in D_+\right)
P(\betaL \in \O^d) \\ 
& = \mathop{\int\dots\int}_{\myatop{m_j \geq z_j}{j \in D_+}} 
\mathop{\int\dots\int}_{\myatop{m_j \leq z_j}{j \in D_-}} 
\mathop{\int\dots\int}_{\myatop{s_j \in [-\lambda_j,\lambda_j]}{j \in D_0}} 
\phi_{(0, \sigma^2X'X)}(X'Xm_\beta + s_\lambda) |X'_{D_\pm}X_{D_\pm}|\,
ds_{D_0} dm_{D_-} dm_{D_+},
\end{align*}
where $m_\beta \in \R^p$ is defined by $(m_\beta)_{D_\pm} = m_{D_-
\cup D_+} $, and $(m_\beta)_{D_0} = -\beta_{D_0}$, and $s_\lambda \in
\R^p$ is defined by $(s_\lambda)_{D_-} = -\lambda_{D_-}$,
$(s_\lambda)_{D_+} = \lambda_{D_+}$, and $(s_\lambda)_{D_0} =
s_{D_0}$. Differentiating with respect to $z_j$ with $j \in D_\pm$,
and taking the absolute value gives the density, thus completing the
proof.
\end{proof}

Besides the conditional densities, we can also specify the full cdf of
$\hat u = \betaL - \beta$ which is done in the following theorem.

\begin{theorem} \label{thm:cdf} 
The cdf of $\hat u = \betaL - \beta$ is given by
$$
F(z) = P(\hat u_1 \leq z_1,\dots,\hat u_p \leq z_p) = \sum_{d \in
\{-1,0,1\}^p}\mathop{\int\dots\int}_{\myatop{m_j \leq z_j \;\;}{j \in
D^+_-}} h^d(m_{D_\pm}) \, d\nu_{\|d\|_1},
$$
where $\nu_k$ denotes $k$-dimensional Lebesgue-measure, and where
$$  
h^d(m_{D_\pm}) = \ind\{m_\beta + \beta \in \O^d\} \!\!
\mathop{\int\dots\int}_{\myatop{s_j \in [-\lambda_j,\lambda_j]}{j \in D_0}} 
\phi_{(0,\sigma^2 X'X)} \left(X'Xm_\beta + s_\lambda\right) 
|X'_{D_\pm}X_{D_\pm}| ds_{D_0},
$$
with $m_\beta \in \R^p$ given by $(m_\beta)_{D_\pm} = m_{D_\pm}$, 
$(m_\beta)_{D_0} = -\beta_{D_0}$, and $s_\lambda \in \R^p$
given by $(s_\lambda)_{D_-} = -\lambda_{D_-}$, $(s_\lambda)_{D_0} =
s_{D_0}$, and $(s_\lambda)_{D_+} = \lambda_{D_+}$.
\end{theorem} 

\begin{proof}
It is easily seen that
$$
P\left(\hat u_1 \leq z_1,\dots,\hat u_p \leq z_p\right) = 
\sum_{d \in \{-1,0,1\}^p} P(\betaL \in \O^d) 
\mathop{\int\dots\int}_{\myatop{m_j \leq z_j \;\;}{j \in D^+_-}} 
f^d(m_{D_\pm}) \, d\nu_{\|d\|_1}.
$$
Plugging in the formula for $f^d$ completes the proof.
\end{proof}

For illustration of Proposition~\ref{prop:density} and
Theorem~\ref{thm:cdf}, consider Figures~\ref{fig:density} and
\ref{fig:conddensity} which display an example of the distribution of
$\hat u = \betaL - \beta$. One can see that the Lasso estimation error
follows a shifted normal distribution, conditional on the event $\hat
u_j \neq - \beta_j$ ($\betaLj \neq 0$) for each $j$, with the shift
depending on the signs of $\betaL$, as can be seen in
Figure~\ref{fig:density}. Figure~\ref{fig:conddensity}  displays the
mass which lies on the set $\{z \in \R^2: z_1=-\beta_1, z_2 \neq 0
\}$, that is, the density functions $h^{(0,1)}$ and $h^{(0,-1)}$ on
their corresponding domains. The mass on the set $\{z \in \R^2: z_1
\neq 0, z_2 = -\beta_2\}$ looks qualitatively similar to
Figure~\ref{fig:conddensity}. Note that we also have point-mass at
$-\beta$, as is pointed out by Corollary~\ref{cor:allzero}.

\begin{figure}[htp]
\centering \includegraphics[width=0.49\textwidth]{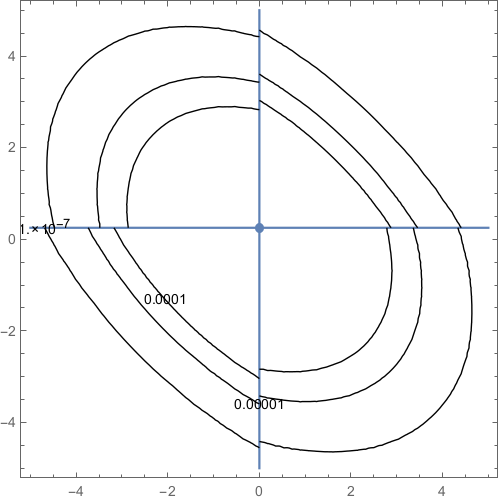}
\caption{\label{fig:density} The contour lines of the absolutely
continuous part of the distribution of $\betaL - \beta$ with respect
to 2-dimensional Lebesgue-measure, for $X'X=
\left(\protect\begin{smallmatrix} 1 & 0.5 \\
0.5 & 1 \protect\end{smallmatrix}\right)$, 
$\lambda=(0.75,0.75)'$, and $\beta=(0, -0.25)'$. Note that the blue
lines as well as the blue point also carry probability mass.}
\end{figure} 

\begin{figure}[htp]
\centering \includegraphics[width=0.49\textwidth]{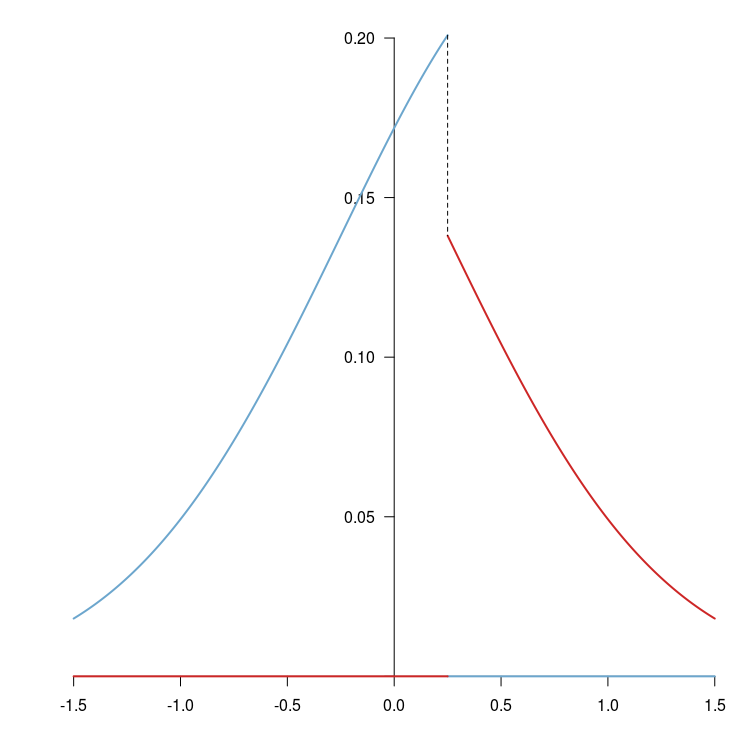}
\caption{\label{fig:conddensity} The functions $h^{(0,-1)'}$ (in blue)
and $h^{(0,1)'}$ (in red) for $X'X=
\left(\protect\begin{smallmatrix} 1 & 0.5 \\ 
0.5 & 1 \protect\end{smallmatrix}\right)$,
$\lambda=(0.75,0.75)'$ and $\beta=(0, -0.25)'$, corresponding to the
absolutely continuous part.}
\end{figure} 
  
\subsection{Shrinkage Areas}
\label{subsec:shrink}

Using the conditions for minimality from Lemma~\ref{lem:kkt}, we can
establish a direct relationship between the LS and the Lasso estimator
in the following sense. For any $b \in \R^p$, there exists a set $S(b)
\subseteq \R^p$, such that the Lasso estimator assumes the value $b$
if and only if the LS estimator lies in $S(b)$. We refer to the set
$S(b)$ as \emph{shrinkage area} since the Lasso estimator can be
viewed as a procedure that shrinks the LS estimates from the set
$S(b)$ to the point $b$. Note that by shrinkage, we mean that $\|b\|_1
\leq \|z\|_1$ for each $z \in S(b)$, but $|b_j| > |z_j|$ could hold
for certain components. The explicit form of $S(b)$ is formalized in
the following theorem.

\begin{theorem} \label{thm:shrink}
For each $b \in \R^p$ there exists a set $S(b) \subseteq \R^p$, such that 
$$
\betaL = b \Longleftrightarrow \betaLS \in S(b).
$$
Moreover, for $b \in \O^d$, the set $S(b)$ is given by
\begin{align*}
S(b) = \{z \in \R^p : & (X'Xz)_j = (X'Xb)_j + \sgn(b_j) \lambda_j \text{ for
} j \in D_\pm, \\
& |(X'X(z-b))_j| \leq \lambda_j \text{ for }  j \in D_0\}.
\end{align*}
Clearly, the sets $S(b)$ are disjoint for different $b$'s.
\end{theorem}
\begin{proof}
Using Lemma~\ref{lem:kkt}, we find that $\betaL = b$ holds if and only
if $X'X\betaLS \in X'Xb + \prod_{j=1}^p B_j(b_j)$, which, for $b \in
\O^d$, holds if and only if
$$
\begin{cases}
\; (X'X\betaLS)_j = (X'X\betaL)_j - \lambda_j & \text{ for } j \in D_- \\
\; |(X'X(\betaLS - \betaL))_j| \leq \lambda_j & \text{ for } j \in D_0 \\
\; (X'X\betaLS)_j = (X'X\betaL)_j + \lambda_j & \text{ for } j \in D_+, 
\end{cases}
$$
or, $\betaLS \in S(b)$, as required.

The sets are disjoint since, in case $\rk(X) = p$, all Lasso solutions
are unique. If $S(b) \cap S(\tilde b) \neq \emptyset$ holds for some
$b \neq \tilde b$, we can find $y \in \R^n$ such that $\betaLS =
(X'X)^{-1}X'y \in S(b) \cap S(\tilde b)$ implying that both $b$ and
$\tilde b$ are Lasso solutions for the given $y$, yielding a
contradiction.
\end{proof}

\begin{remark} \label{rem:singletonLOW}
Clearly, if $b \in \R^p$ satisfies $b_j \neq 0$ for all $j=1,\dots,p$, then
$S(b)$ is the singleton 
$$
S(b) = \{b + (X'X)^{-1}\tilde\lambda\},
$$ where
$\tilde\lambda_j = \sgn(b_j)\lambda_j$ for $j = 1,\dots,p$. This
implies that, in case $\betaLj \neq 0$ for all $j$, the Lasso
estimator is given by
$$
\betaL = \betaLS + (X'X)^{-1}\tilde\lambda.
$$
Note that aside from $b$, $S(b)$ depends on $X$ and $\lambda$ only. 
\end{remark}

Given Theorem~\ref{thm:shrink}, we can identify areas in which
components of the LS estimator are shrunk to zero by the Lasso. For
$p=2$, it leads to the image displayed in Figure~\ref{fig:shrink}.
Clearly, the shrinkage areas are related to the polyhedral selection
areas developed in \cite{TibshiraniTaylor12} and employed for instance
in \cite{LeeEtAl16}, but yield a different kind of information. Our
results identify the regions of the LS estimator that lead to a
particular value $b$ of the Lasso estimator. The polyhedral regions in
the above articles identify the regions of the dependent variable $y$
that correspond to a particular Lasso model with specific signs of the
active coefficients. Naturally, our regions are subsets of $\R^p$,
while the polyhedral regions are subsets of $\R^n$ (the latter ones
also allowing to interpret the Lasso fit as a projection, and not
depending on full column rank).

\begin{figure}[htp]
\centering \includegraphics[width=0.49\textwidth]{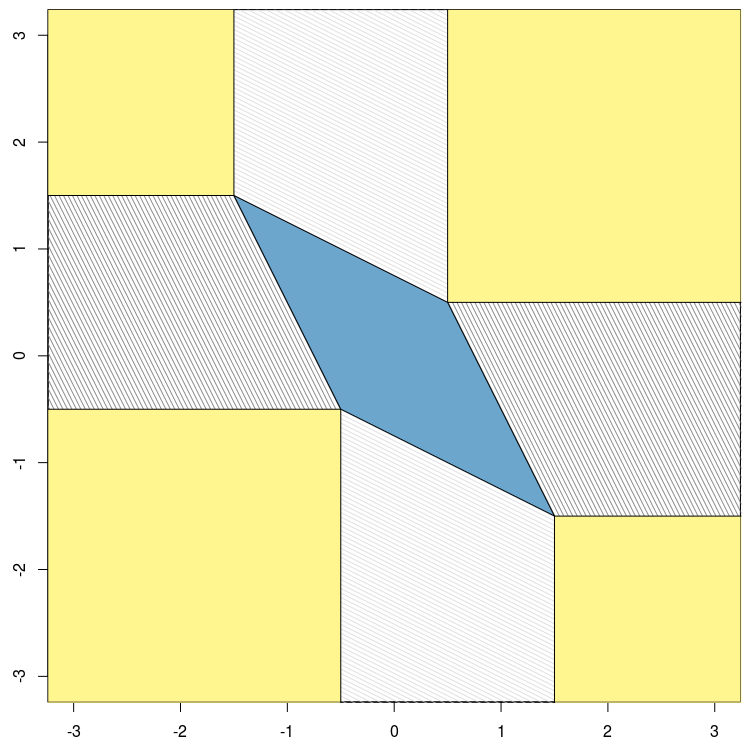}
\caption{\label{fig:shrink} The shrinkage areas from
Theorem~\ref{thm:shrink} for $p=2$. The blue parallelogram equals the
set $S(0)$. The dark gray area consists of lines parallel to the
adjacent edge of the parallelogram, where each line
equals a set $S\left(\protect\begin{smallmatrix} 0 \\ b_2 
\protect\end{smallmatrix}\right)$ 
for $b_2 \neq 0$. Analogously, the light gray area consists of lines
parallel to the adjacent edge of the parallelogram, and each of those
lines equals a set
$S\left(\protect\begin{smallmatrix} b_1 \\ 0 \protect\end{smallmatrix} \right)$ 
for $b_1 \neq 0$. The yellow areas contain all singletons
$S\left(\protect\begin{smallmatrix} b_1 \\ b_2 \protect\end{smallmatrix}\right)$,
with $b_1, b_2 \neq 0$ as described in Remark~\ref{rem:singletonLOW}.
In this example, $X'X= \left(\protect\begin{smallmatrix} 1 & 0.5 \\
0.5 & 1 \protect\end{smallmatrix}\right)$, and $\lambda=(0.75,0.75)'$.}
\end{figure} 

\section{High-Dimensional Case}
\label{sec:HIGH}

We now turn to the main case of this this article, the
high-dimensional setting where $p > n$. \emph{We make no assumptions
on the regressor matrix $X$ in this section}. Using similar arguments
as in the case $p \leq n$, we can again start by characterizing the
distribution of the Lasso, albeit in a somewhat less explicit form.
Note that we have $\rk(X) < p$ and that the true parameter is not
identified without further assumptions. We denote by $\Bzero$ the set
of all $\beta \in \R^p$ that yield the model given in
\eqref{eqn:model}, that is, $\Bzero = \{\beta \in \R^p : X\beta =
\E(y) = \mu\}$. Furthermore, it is important to note that the Lasso
solution need not be unique anymore. We give necessary and sufficient
conditions for uniqueness later in Section~\ref{subsec:unique}.

All findings in this section also hold when $p \leq n$, but more
explicit results for this case are found in Section~\ref{sec:LOW}. We
start with a high-level result on the distribution of $\betaL$, which
immediately follows from Corollary~\ref{cor:kkt}.

\begin{theorem} \label{thm:probHIGH}
For any set $B \subseteq \R^p$ and any $\beta \in \Bzero$, we have
$$
P(\argmin_{\upbeta \in \R^p} L(\upbeta) \cap B \neq \emptyset) = 
P(W \in A_\beta(B)),
$$
where $W \sim N(0,\sigma^2X'X)$, and $A_\beta(B) = \bigcup_{b \in B} A_\beta(b)$,
with $A_\beta(b) = X'X(b-\beta) + \prod_{j=1}^p B_j(b_j)$, and
$$
B_j(b_j) = 
\begin{cases} 
\{\sgn(b_j) \lambda_j\} & b_j \neq 0 \\
[-\lambda_j, \lambda_j] & b_j = 0.
\end{cases}
$$
In particular, the distribution of the estimator $\betaL$ does not depend
on the choice of $\beta \in \Bzero$.
\end{theorem}

To derive the analogue of the distribution of the estimation error in
high dimensions for a fixed $\beta \in \Bzero$, define the function
$V_\beta(u) = L(u + \beta) - L(\beta)$ given by
\begin{equation} \label{eqn:Vbeta}
V_\beta(u) = L(u+\beta) - L(\beta) = u'X'X u - 2 u'W  
+ 2\sum_{j=1}^p \lambda_{j} \left[ |u_j + \beta_j|-|\beta_j| \right],
\end{equation}
which is minimized at $\betaL - \beta$, and where $\betaL$ may be any
minimizer of $L(\upbeta)$. The high-dimensional version of
Corollary~\ref{cor:prob} can now be formulated as

\begin{theorem} \label{cor:probHIGH}
For any set $M \subseteq \R^p$ and any $\beta \in \Bzero$, we have
$$
P(\argmin_{u \in \R^p} V_\beta(u) \cap M \neq \emptyset) = 
P(W \in \bar A_\beta(M)),
$$
where $W \sim N(0,\sigma^2X'X)$ and $\bar A_\beta(M) = \bigcup_{m \in
M} \bar A_\beta(m)$ with $\bar A_\beta(m) = X'Xm + \prod_{j=1}^p
\bar B_{\beta,j}(m_j)$ and
$$
\bar B_{\beta,j}(m_j) =
\begin{cases} 
\{\sgn(m_j + \beta_j) \lambda_j\} & m_j + \beta_j \neq 0 \\
[-\lambda_j, \lambda_j] & m_j + \beta_j = 0.
\end{cases}
$$
\end{theorem}

\begin{proof}
As stated above, $m \in \R^p$ is a minimizer of $V_\beta$ if and only
if $m + \beta$ is a minimizer of $L(\upbeta)$. Corollary~\ref{cor:kkt}
then yields
$$
m \in \argmin_{u \in \R^p} V_\beta(u) \iff W \in A_\beta(m+\beta) \iff
W \in \bar A_\beta(m).
$$
\end{proof}

\begin{remark} \label{rem:normHIGH}
Note that, just as for the low-dimensional case discussed in
Remark~\ref{rem:normLOW}, the statements in Theorem~\ref{thm:probHIGH}
and Corollary~\ref{cor:probHIGH} do not hinge on the normal
distribution of $W = X'\eps$. In fact, the both results equally hold
for arbitrary distributions of $W$.
\end{remark}

While the distribution of $\betaL - \beta$ depends on the choice of
$\beta \in \Bzero$, the distribution of $\betaL$ does not, as it is
determined by $y \sim N(\mu,\sigma^2 I_n)$. This is further formalized
in the following corollary. As mentioned before, $\betaL$ need not be
unique. Also remember that $\betaL$ itself minimizes the function
$L(\upbeta)$ defined in \eqref{eqn:lassodef}.

As the random variable $W = X'\eps$ has singular covariance matrix,
some care needs to be taken when computing the probability from
Corollary~\ref{cor:probHIGH} through the appropriate integral of the
corresponding density function.

\begin{corollary} \label{cor:intrepHIGH}
Let the columns of $U$ form a basis of $\col(X')$. The probability
that a Lasso solution lies in the set $B \subseteq \R^p$ can be
written as
$$
P(\argmin_{\upbeta \in \R^p} L(\upbeta) \cap B \neq \emptyset) =
\ind\{\col(X')\cap A_\beta(B) \neq \emptyset\} \int_{U'A_\beta(B)}
\phi_{(0,\sigma^2 U'X'X U)}(w)dw.
$$
\end{corollary}

\begin{proof}
Note that $U'W \sim N(0,\sigma^2U'X'XU)$ and that $U'X'XU$ is
invertible. Let $N$ be a matrix whose columns form a basis of
$\col(X')^\perp$, so that $N'W$ has covariance matrix $\sigma^2N'X'XN
= 0$, yielding $N'W = 0$ almost surely. We therefore have
\begin{align*}
W \in A_\beta(B) & \iff (U,N)'W \in (U,N)'A_\beta(B) \\
& \iff U'W \in U'A_\beta(B) \text{ and } 0 \in N'A_\beta(B) \\
& \iff U'W \in U'A_\beta(B) \text{ and } \col(X') \cap A_\beta(B) \neq \emptyset,
\end{align*}
which proves the claim.
\end{proof}

\subsection{Selection Regions and Model Selection Properties}
\label{subsec:shrinkHIGH}

In the low-dimensional case, Theorem~\ref{thm:shrink} gives what we
call shrinkage areas of the Lasso with respect to the LS estimator. As
the latter is never uniquely defined in the high-dimensional case, we
instead look at the object $X'y$ and and consider so-called
\emph{selection regions} with respect to this quantity: for any $b \in
\R^p$, we provide a set $T(b)$ such that a Lasso solution is equal to
$b$ if and only if $X'y$ lies in the set $T(b)$. The corresponding
result turns out to be a restatement of Lemma~\ref{lem:kkt}, which we
list again in the following for the sake of completeness.

\begin{theorem} \label{thm:shrinkHIGH} 
For each $b \in \R^p$ there exists a set $T(b) \subseteq \R^p$
such that
$$
b \in \argmin_{\upbeta \in \R^p} L(\upbeta) \iff X'y \in T(b).
$$
Moreover, $T(b)$ is given by
$$
T(b) = X'Xb + \prod_{j=1}^p B_j(b_j)
$$
with
$$
B_j(b_j) = 
\begin{cases} 
\{\sgn(b_j) \lambda_j\} & b_j \neq 0 \\
[-\lambda_j, \lambda_j] & b_j = 0.
\end{cases}
$$
\end{theorem}

\begin{remark} \label{rem:singletonHIGH}
Analogously to the low-dimensional case, the sets $T(b)$ are
singletons if $b \in \R^p$ satisfies $b_j \neq 0$ for all
$j=1,\dots,p$:
$$
T(b) = \{X'Xb + \tilde\lambda\},
$$ 
where $\tilde\lambda_j = \sgn(b_j)\lambda_j$ for $j = 1,\dots,p$.
Also, aside from $b$, the sets $T(b)$ depend on $X$ and $\lambda$
only.
\end{remark}

Inspecting the sets $T(b)$ from Theorem~\ref{thm:shrinkHIGH} more
closely, we see that they are, in general, not disjoint for different
values of $b \in \R^p$. This illustrates the fact that, in contrast to
the low-dimensional case, the Lasso solution need not be unique in
high dimensions anymore. Indeed, we can have $T(b) \cap T(b') \neq
\emptyset$, as long as $b - b' \in \ker(X)$ and
$\{\sgn(b_j),\sgn(b_j')\} \neq \{-1,1\}$ for all $j$. This also makes
apparent that $b$ and $b'$ may be Lasso solutions not corresponding to
the same model, which has been noted by \cite{Tibshirani13} for the
case of $\lambda_1 = \dots  = \lambda_p > 0$. We get deeper into the
issue of (non-)uniqueness in Section~\ref{subsec:unique}.

Theorem~\ref{thm:shrinkHIGH} also sheds some light on which models $\M
\subseteq \{1,\dots,p\}$ may in fact be chosen by the Lasso estimator,
where the Lasso model is given by $\{j : \betaLj \neq 0\}$. We find
that some models will, in fact, never be selected by the Lasso. This
is illustrated in Figure~\ref{fig:shrinkHIGH} below, where the Lasso
always sets the first component to zero, independently of $y$. This
leads to the question on how to determine whether a particular model
$\M$ may or may not be chosen.

\begin{figure}[htp]
\centering \includegraphics[width=0.49\textwidth]{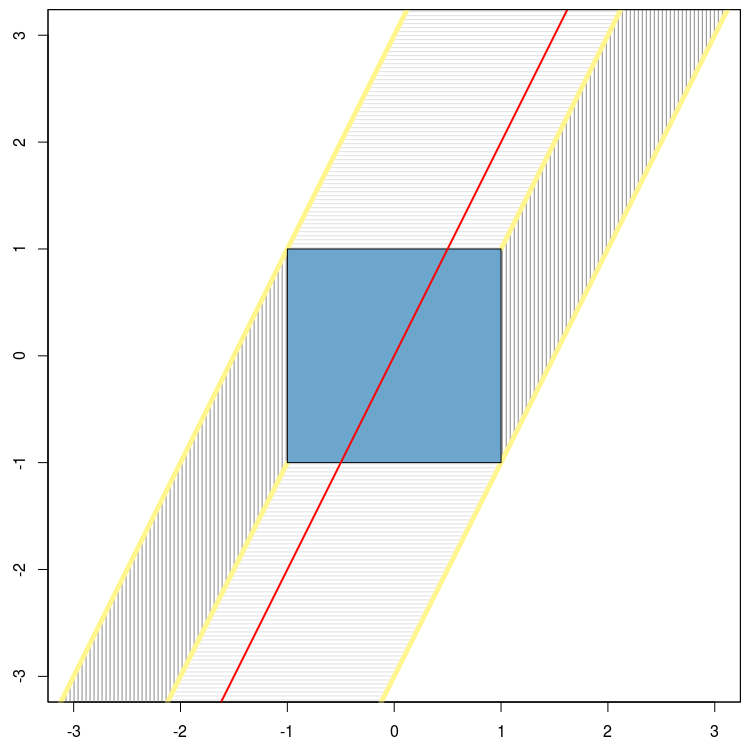}
\caption{\label{fig:shrinkHIGH} The selection regions with respect to
$X'y$ from Theorem~\ref{thm:shrinkHIGH}, with $X = (1 , 2)$ and
$\lambda_1 = \lambda_2 = 1$ from Example~\ref{ex:n1p2}. Displayed in
red is $\col(X')$, the area on which the probability mass of $X'y$ is
concentrated. The set
$T\left(\protect\begin{smallmatrix} 0 \\ 0\protect\end{smallmatrix}\right)$ 
is displayed in blue, while the parallel light gray lines represent
the sets
$T\left(\protect\begin{smallmatrix} 0 \\ b_2\protect\end{smallmatrix}\right)$
with $b_2 \neq 0$, and the parallel dark gray lines are the sets
$T\left(\protect\begin{smallmatrix} b_1 \\ 0\protect\end{smallmatrix}\right)$
with $b_1 \neq 0$. The yellow lines consist of the singletons
$T\left(\protect\begin{smallmatrix} b_1 \\ b_2\protect\end{smallmatrix}\right)$ 
with $b_1, b_2\neq 0$. Note that the red line does not intersect any of the sets
$T\left(\protect\begin{smallmatrix} b_1 \\ 0\protect\end{smallmatrix}\right)$ or
$T\left(\protect\begin{smallmatrix} b_1 \\ b_2\protect\end{smallmatrix}\right)$ 
with $b_1, b_2 \neq 0$.}
\end{figure}

Along these lines, define $\BM = \{b \in \R^p: b_j \neq 0 \text{ if
and only if } j \in \M\}$. Then there exists a $y \in \R^n$ such that
a corresponding corresponding Lasso solution chooses $\M$ if and only
if there exists $y \in \R^n$ such that $X'y \in T(\BM)$. In other
words, this is the case if and only if $\col(X') \cap T(\BM) \neq
\emptyset$ with $T(\BM) = \bigcup_{b \in \BM} T(b)$. Looking at the
definition of $T(b)$ in Theorem~\ref{thm:shrinkHIGH}, and noting that
$X'Xb \in \col(X')$, we can deduce the following corollary.

\begin{corollary} \label{cor:modelselection}
Let $X \in \R^{n \times p}$ and $\lambda \in \Rgeq^p$ be given. There
exist $y \in \R^n$ such that a corresponding Lasso solution selects
model $\M \subseteq \{1,\dots,p\}$ if and only if
$$
\col(X') \cap \B_\M \neq \emptyset,
$$
where 
$$
\B_\M = \prod_{j=1}^p \begin{cases} 
\{-\lambda_j, \lambda_j\} & j \in \M \\ 
[-\lambda_j, \lambda_j] & j \notin \M,
\end{cases}
$$
which satisfies $\B_{\tilde\M} \subseteq \B_\M$ for $\M \subseteq
\tilde\M$.
\end{corollary}

A model that may be chosen by the Lasso is called \emph{accessible} in
\cite{SepehriHarris17}. (This reference who also provides a condition
for when this is the case. The difference is that uses geometric
considerations in $\R^n$ under a uniqueness assumption, whereas our
approach operates in $\R^p$ with no assumptions on $X$.)

The sets $\B_\M$ are made up of the faces of the $\lambda$-cube. If
$\M_0 = \emptyset$, $\B_\emptyset$ is the $p$-dimensional
$\lambda$-box, $B_{\{j\}}$ is the union of two opposite facets of the
$\lambda$-box, and for $1 < |\M| < p$, $\B_\M$ is a union of
$(p-|\M|)$-dimensional faces of the $\lambda$-box. Finally,
$\B_{\{1,\dots,p\}}$ simply contains the corners of the $\lambda$-box.
These sets are illustrated in Figure~\ref{fig:BMsets} below.

\begin{figure}[htp]
\centering \includegraphics[width=0.49\textwidth]{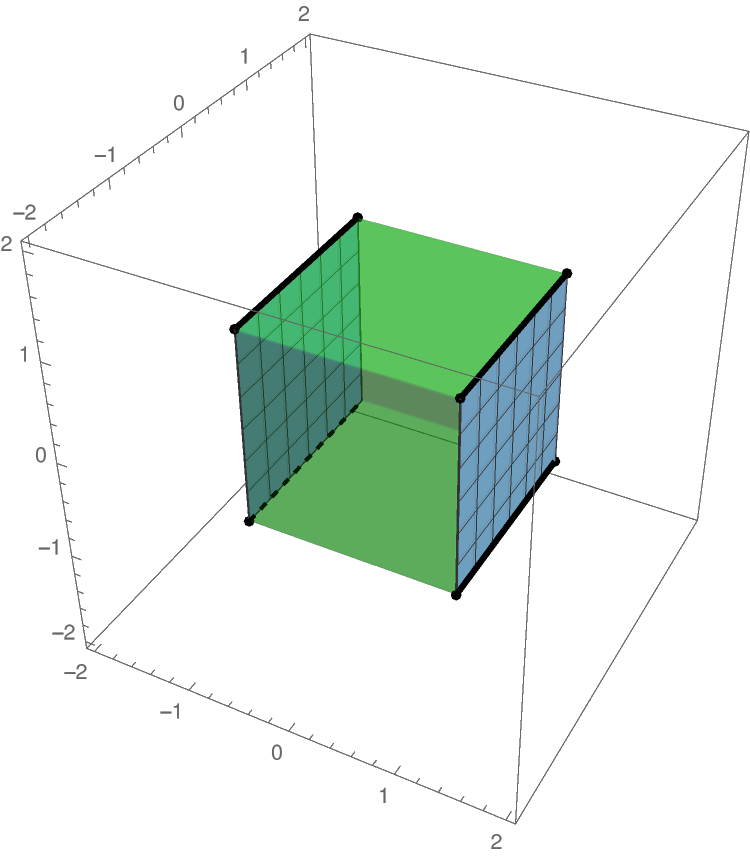}
\caption{\label{fig:BMsets} Illustration of some of the sets $\B_\M$
for $p=3$ and $\lambda_1 = \lambda_2 = \lambda_3 =1$. The set
$\B_{\{1,2,3\}}$ is given by the eight corners of the unit cube,
whereas $\B_{\{1,3\}}$ contains the four parallel edges shown.
$\B_{\{1\}}$ is the union of the two parallel 2-dimensional facets
highlighted by a grid. The sets $\B_{\{1,2\}}$ and $\B_{\{2,3\}}$ are
not depicted, but contain the four parallel vertical and horizontal
edges of the cube, respectively. The sets $\B_{\{2\}}$ and
$\B_{\{3\}}$, which are not shown either, each consist of the
remaining two parallel vertical and horizontal facets.}
\end{figure} 

For partial tuning with $\M_0 \neq \emptyset$, $\B_\emptyset$ is $(p -
|\M_0|)$-dimensional and we have $\B_\M \subseteq \B_{\M_0}$ for all
$\M \subseteq \{1,\dots,p\}$ as well as
$$
\{0\} \subseteq \col(X') \cap \B_{\M_0} \neq \emptyset,
$$
so that, not surprisingly, there always exist $y$ such that the
non-penalized components will be part of the model chosen by the Lasso
solution.

\begin{example} \label{ex:n1p2}
Suppose $X = (1,2)$, so that $n=1$, $p=2$ and let $\lambda_1 =
\lambda_2 = \lambda_3$ (uniform tuning). As can be learned from
Figure~\ref{fig:shrinkHIGH},
$$
\col(X') \cap \B_{\{1\}} = \emptyset
$$
for all $\bar\lambda > 0$, so that by
Corollary~\ref{cor:modelselection}, $\betaLx{1} = 0$ for any value of
$y$, independent of $\Bzero$ and $\sigma^2$. The distribution of the
remaining component $\betaLx{2}$ is now given by a soft-thresholding
rule (also see the paragraph following Theorem~\ref{thm:prob}).
\end{example}

\begin{example} \label{ex:n2p3}
To look at a more complex example, suppose now that
$$
X = \left(\begin{matrix} 
2 & 0 & 1\\
1 & 2 & 0 
\end{matrix} 
\right),
$$ 
so that $n=2$, $p=3$. Let $\lambda_1 = \lambda_2 = \lambda_3 =
\bar\lambda$ (uniform tuning). We have
$$
\col(X') \cap \B_{\{3\}} = \emptyset
$$
for all $\bar\lambda > 0$, so that by Corollary~\ref{cor:modelselection}
$$
P(\betaLx{3} = 0) = 1.
$$ 
To say something about the distribution of the remaining components,
note that the estimator is equivalent to the low-dimensional procedure
using the matrix $\tilde X = X_{\{1,2\}}$, which contains the first
and second regressor only. Let $\tilde\beta \in \R^2$ be such that
$\tilde X\tilde\beta = X\beta$, where $X\beta = E(y)$, and let
$$
V = (V_1,V_2)' \sim N(\tilde\beta - (\tilde X'\tilde
X)^{-1}\tilde\lambda,\sigma^2(\tilde X'\tilde X)^{-1}),
$$ 
where $\tilde\lambda$ will be specified below. We can now use
Theorem~\ref{thm:prob} to find the following. The absolutely
continuous parts of the distribution of $(\betaLx{1},\betaLx{2})'$ can
be determined by
\begin{align*}
P(\betaLx{1} \leq & z_1, \betaLx{2} \leq z_2, \betaLx{3} = 0) \\
& = \int_{-\infty}^{z_2 - \tilde\beta_2} \int_{-\infty}^{z_1 - \tilde\beta_1} 
\phi_{(0,\sigma^2 \tilde X'\tilde X)} (\tilde X'\tilde Xm + \tilde\lambda) \, |\tilde X'\tilde X|
\, dm_1 dm_2 \\
& = P(V_1 \leq z_1, V_2 \leq z_2)
\end{align*}
for $z_1, z_2 < 0$ and $\tilde\lambda = (-\bar\lambda,-\bar\lambda)'$.
Analogously, we get
$$
P(\betaLx{1} \geq z_1, \betaLx{2} \geq z_2, \betaLx{3} = 0) = P(V_1 \geq z_1, V_2 \geq z_2)
$$
for $z_1, z_2 > 0$ and $\tilde\lambda = (\bar\lambda,\bar\lambda)'$. Moreover, 
$$
P(\betaLx{1} \leq z_1, \betaLx{2} \geq z_2, \betaLx{3} = 0) = P(V_1 \leq z_1, V_2 \geq z_2)
$$
for $z_1 < 0, z_2 > 0$ and $\tilde\lambda = (-\bar\lambda,\bar\lambda)'$, as well as 
$$
P(\betaLx{1} \geq z_1, \betaLx{2} \leq z_2, \betaLx{3} = 0) = P(V_1 \geq z_1, V_2 \leq z_2)
$$
for $z_1 > 0, z_2 < 0$ and $\tilde\lambda = (\bar\lambda,-\bar\lambda)'$. 

This shows that the absolutely continuous parts of the estimator
follow a normal distribution with the same covariance matrix as the LS
estimator and a shift in expectation that depends on the regressor
matrix as well as the tuning parameters. These findings are in line
with Remark~\ref{rem:singletonLOW}.

The pointmass part of $(\betaLx{1},\betaLx{2})'$ at $(0,0)'$ has weight
$$
P(\betaLx{1} = \betaLx{2} = \betaLx{3} = 0) 
= \int_{-\bar\lambda}^{\bar\lambda} \int_{-\bar\lambda}^{\bar\lambda} 
\phi_{(\tilde X'\tilde X\tilde\beta,\sigma^2 \tilde X'\tilde X)} (w) dw_1 dw_2.
$$
For the remaining ``mixed'' terms, let 
$$
\tilde X'\tilde X = \left(\begin{matrix} a & b \\ b & c \end{matrix}\right)
$$
and 
$$
Z = (Z_1,Z_2)' \sim N(\mu,\sigma^2 \Sigma) \text{ with } 
\Sigma =  \left(\begin{matrix} \sigma_1^2 & 0 \\ 0 & \sigma_2^2\end{matrix}\right).
$$
After some calculations, it can be shown that 
\begin{align*}
P(\betaLx{1} & = 0, \betaLx{2} \leq z_2, \betaLx{3} = 0) \\
& = \int_{-\infty}^{z_2 - \tilde\beta_2} \int_{-\tilde\lambda}^{\tilde\lambda} 
\phi_{(0,\sigma^2 \tilde X'\tilde X)} (\tilde X'\tilde X(-\tilde\beta_1,m_2)' + (s_1,-\bar\lambda)')
\, c \, ds_1 dm_2 \\
& = P(- \bar\lambda \leq Z_1 \leq \bar\lambda, Z_2 \leq z_2) 
= P(- \bar\lambda \leq Z_1 \leq \bar\lambda)P(Z_2 \leq z_2)
\end{align*}
for $z_2 < 0$ and $\mu_1 = (|\tilde X'\tilde X|\tilde\beta_1 -
b\bar\lambda)/c$, $\sigma_1^2 = |\tilde X'\tilde X|/c$, $\mu_2 =
(b\tilde\beta_1 + \bar\lambda)/c + \beta_2$, and $\sigma_2^2 = 1/c$.

We get an analogous expression for $P(\betaLx{1} = 0, \betaLx{2} \geq
z_2, \betaLx{3} = 0)$ for $z_2 > 0$, but with $\bar\lambda$ replaced
by $-\bar\lambda$ in $\mu_1$ and $\mu_2$. Moreover,
$$
P(\betaLx{1} \leq z_1, \betaLx{2} = \betaLx{3} = 0) = 
P(Z_1 \leq z_1)P(-\bar\lambda \leq Z_2 \leq \bar\lambda)
$$
for $z_1 < 0$ and $\mu_1 = \tilde\beta_1 + (b\tilde\beta_2 +
\bar\lambda)/a$, $\sigma_1^2 = 1/a$, $\mu_2 = (|\tilde X'\tilde
X|\tilde\beta_2 - b\bar\lambda)/a$, $\sigma_2^2 = |\tilde X'\tilde
X|/a$.

Finally, we get an analogous expression for $P(\betaLx{1} \geq z_1,
\betaLx{2} = \betaLx{3} = 0)$ for $z_1 > 0$, but with $\bar\lambda$
replaced by $-\bar\lambda$ in $\mu_1$ and $\mu_2$. It might be
interesting to note that in this example, the probabilities for the
distribution that are made up of both a pointmass part and absolutely
continuous parts can be represented by independent (normal) random
variables.
\end{example}

In both examples above, the distribution of $\betaL$ is the same as
the one of a Lasso estimator in a smaller model. This fact is, of
course, only valid for the specific forms of $X$ and $\lambda$
considered here. The models considered by the Lasso do not depend on
$\beta$ and $\eps$ in the sense that certain values of $X$ and
$\lambda$ may immediately rule out certain models, completely
independently of $y$. (The choice between the accessible models does,
of course, very much depend on $\beta$ and $\eps$.)

Sparked by Examples~\ref{ex:n1p2} and \ref{ex:n2p3}, this suggests
that in the high-dimensional setting, model selection by the Lasso
estimator may possibly not be a purely data-driven procedure insofar
as there is a \emph{structural model} or \emph{structural set}
$\Mstruct \subseteq \{1,\dots,p\}$, determined by $X$ and $\lambda$
only, that satisfies $\betaLj = 0$ for any $j \notin \Mstruct$
\emph{for all observations} $y \in \R^n$. In particular, the true
parameter $\beta \in \Bzero$, as well as the distribution of $\eps$ do
not have any influence on this set. In other words, some models are
never considered by the model selection procedure, completely
independently of the data vector $y$. Put yet differently again, for a
given regressor matrix $X$, one can restrict or choose this class of
models by choice of $\lambda$.

Given all the considerations above, one might ask whether such a
structural model $\Mstruct$ always satisfies $|\Mstruct| \leq n$ under
certain conditions. Clearly, uniqueness would be a meaningful
requirement in this context, as then all Lasso solutions will choose
models of cardinality of at most $n$, as has been shown in
\cite{Tibshirani13}\footnote{Note that this fact alone does not imply
that the structural set has cardinality of at most $n$ since the
active sets may certainly vary over $y$.}. In that case, the Lasso
estimator would be equivalent to a low-dimensional Lasso procedure,
restricted to this structural model $\Mstruct$, and we could employ
results from low-dimensional settings also for inference in
high-dimensional models, such as \cite{EwaldSchneider18} for
constructing confidence regions.


In Examples~\ref{ex:n1p2} and \ref{ex:n2p3}, the Lasso solutions are
always unique. It is not difficult, however, to construct an example
where the solutions are not unique anymore.

\begin{example} \label{ex:n1p2NONunique}
Again, take the model from Example~\ref{ex:n1p2} with $X=(1,2)$. This
time, choose $\lambda=(1,2)'$ (non-uniform tuning). It can easily be
seen using Theorem~\ref{thm:probHIGH} that for each $y < -1$,
$$
\betaL = \left(\begin{smallmatrix} y+1 \\0 \end{smallmatrix}\right), \;
\betaL = \left(\begin{smallmatrix} 0 \\ \frac{y+1}{2} \end{smallmatrix}\right), \;
\text{ and } \;
\betaL = \left(\begin{smallmatrix} y+1-2c \\ c \end{smallmatrix}\right) \; 
\text{ for } \; (y+1)/2 < c < 0
$$
all are Lasso solutions for the same value of $y$. Similarly, for $y > 1$,
$$
\betaL = \left(\begin{smallmatrix} y-1 \\0 \end{smallmatrix}\right), \;
\betaL = \left(\begin{smallmatrix} 0 \\ \frac{y-1}{2} \end{smallmatrix}\right), \;
\text{ and } \;
\betaL = \left(\begin{smallmatrix} y+1-2c \\ c \end{smallmatrix}\right) \; 
\text{ for } \; 0 < c < (y+1)/2
$$
all are Lasso solutions for the same value of $y$. (Note that $\betaL
= 0$ for all $y$ with $|y| \leq 1$.) The corresponding selection
regions are illustrated in Figure~\ref{fig:shrinkHIGH2}.

\end{example}
Example~\ref{ex:n1p2NONunique} shows an already known property of the
Lasso from another perspective: The solution to the Lasso problem is,
in general, not unique. Moreover, if the solution is not unique, then,
by convexity of the problem, there exists an uncountable set of
solutions\footnote{This fact has been pointed out by
\cite{Tibshirani13} in Lemma 1 for the case of uniform tuning.}. The
example moreover shows that the set of $y$ which yield non-unique
Lasso solutions is not a null set. In fact, in this example, it occurs
with probability $2\Phi(-1)$. 

\begin{figure}[htp]
\centering \includegraphics[width=0.47\textwidth]{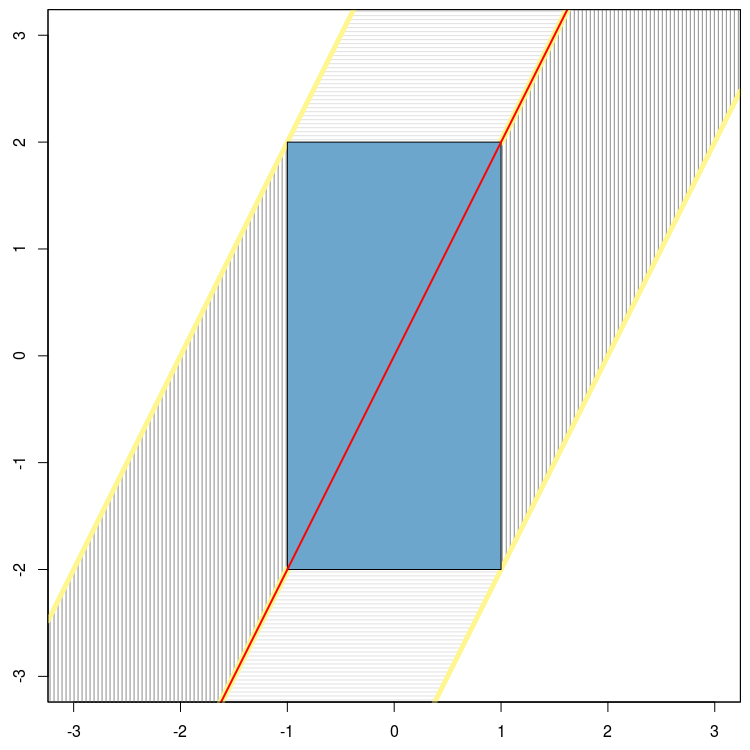}
\caption{\label{fig:shrinkHIGH2} The selection regions with respect to
$X'y$ from Theorem~\ref{thm:shrinkHIGH}, with $X = (1 , 2)$ and
$\lambda_1 = 1$, $\lambda_2 = 2$ from Example~\ref{ex:n1p2NONunique}.
Displayed in red is $\col(X')$, the area on which the probability mass
of $X'y$ is concentrated. The set
$T\left(\protect\begin{smallmatrix} 0 \\ 0\protect\end{smallmatrix}\right)$ 
is displayed in blue, while the parallel light gray lines represent 
$T\left(\protect\begin{smallmatrix} 0 \\ b_2\protect\end{smallmatrix}\right)$
with $b_2 \neq 0$, and the parallel dark gray lines are
$T\left(\protect\begin{smallmatrix} b_1 \\ 0\protect\end{smallmatrix}\right)$
with $b_1 \neq 0$. The yellow lines consist of the singletons 
$T\left(\protect\begin{smallmatrix} b_1 \\ b_2\protect\end{smallmatrix}\right)$ 
with $b_1, b_2\neq 0$. The red line passes through
$T\left(\protect\begin{smallmatrix} 0 \\
0\protect\end{smallmatrix}\right)$ 
where the solution is unique but also through the line where the light
gray, the dark gray and the yellow areas intersect.}
\end{figure} 

Of course, this problem could be overcome by slightly altering the
choice of the tuning parameters, even though this would imply to make
a choice of the class of models under consideration, as pointed out
previously in this section.

\subsection{Structural Sets} 
\label{subsec:struct}

Clearly, Example~\ref{ex:n1p2NONunique} shows that the structural set
may be equal to the entire set of explanatory variables. It is easy to
see that for $n=1$ and $p=2$, the Lasso estimator will always have a
structural set with cardinality $n=1$ whenever we have uniqueness. The
question is, of course, whether the same can be said in more
generality. Before answering this question, we show how the structural
set can be determined given $X$ and $\lambda$, by counting how many
sets of parallel facets of the $\lambda$-box $\B_\emptyset$ are
intersected by $\col(X')$.

\begin{theorem} \label{thm:struct}
Let $X \in \R^{n \times p}$ and $\lambda \in \Rgeq^p$ be given. Let
$\Mstruct$ be the so-called structural set of $X$ and $\lambda$ that
contains all $j \in \{1,\dots,p\}$, such that there exist $y \in \R^n$
so that a corresponding Lasso solution $\betaL$ satisfies $\betaLx{j}
\neq 0$, that is, $\Mstruct$ contains all regressors that are part of
a Lasso solution for some observation $y$. This set is given by
$$
\Mstruct = \Mstruct(X,\lambda) = \big\{j \in \{1,\dots,p\}:
\col(X')  \cap \B_{\{j\}} \neq \emptyset\big\}.
$$
\end{theorem}
\begin{proof}
By Corollary~\ref{cor:modelselection}, there exist $y \in \R^n$ such
that the corresponding Lasso solution chooses model $\M$ if and only
if $\ \col(X') \cap B_\M \neq \emptyset$. For any $\{j\} \subseteq \M$,
we have, by Corollary~\ref{cor:modelselection} also, $\B_\M \subseteq
\B_{\{j\}}$, so that $\B_{\{j\}} \cap \col(X') \neq \emptyset$ also.
\end{proof}

Theorem~\ref{thm:struct} shows that in order to determine the
structural set, only the intersection of $\col(X')$ with the
$(p-1)$-dimensional faces, the so-called facets of the $\lambda$-cube,
have to be considered. A strategy how to determine the structural set
for a given $X$ might be the following. Note that $\col(X') =
\ker(X)^\perp$ and find vectors $V_1,\dots,V_k \in \R^p$ that span
$\ker(X)$, where $k = p - \rk(X)$. Let $V = (V_1,\dots,V_k) \in \R^{p
\times k}$ and check whether $V's = 0$ is solvable for $s \in
\B_{\{j\}}$, where
$$
\B_{\{j\}} = [-\lambda_1,\lambda_1] \times \dots \times \{-\lambda_j,\lambda_j\} 
\times \dots \times [-\lambda_p,\lambda_p].
$$  
If this is the case, then $j \in \Mstruct$, otherwise $j \notin
\Mstruct$. So, determining the structural set amounts to identifying a
basis of $\ker(X)$, and solving a linear system in $k = p - \rk(X)$
equations and $p$ unknowns. After that, we have to check whether the
resulting solution set contains any elements of $\B_{\{j\}}$ for $j =
1,\dots,p$. This approach is employed in Example~\ref{ex:nonGP} in the
subsequent section.

We would like to point out the difference between the idea of a
structural set and results concerning SAFE rules, such as
\cite{ElGhaouiEtAl12}, \cite{TibshiraniEtAl12} and
\cite{NdiayeEtAl17}. Based on a SAFE rule, a regressor will be
discarded by a Lasso solution for a given observation $y$. In
contrast, if a covariate is not contained the the structural set, it
will be excluded from the Lasso model for all observations $y$. This,
on the one hand, implies that the result of Theorem~\ref{thm:struct}
is much cruder than a safe rule, excluding (if any) less regressors.
On the other hand, since the structural set is entirely independent of
$y$, the corresponding Lasso problem can equivalently be viewed as a
Lasso problem using covariates from $\Mstruct$ only, also regarding
distributional results and inference. In particular, if $|\Mstruct|
\leq n$, we can consider the low-dimensional Lasso problem using
$X_{\Mstruct}$ as regressor matrix. If $X_{\Mstruct}$ has full rank,
we can then use results from \cite{EwaldSchneider18} to construct
confidence regions. One has, of course, to be aware that inference is
now on the parameter $\tilde\beta$ satisfying $X\beta =
X_{\Mstruct}\tilde\beta$, as exemplified in Example~\ref{ex:n2p3}.

\begin{remark}
As indicated in Theorem~\ref{thm:struct} and as discussed above, the
structural set $\Mstruct$ depends on $X$ and $\lambda$ only. Moreover,
it can easily be seen that it depends on the tuning parameters
$\lambda$ only through the penalization weighting in the sense that
whenever $\lambda = \bar\lambda \omega$ for some $\bar\lambda > 0$ and
$\omega \in \Rgeq^p$, $\Mstruct(X,\lambda) = \Mstruct(X,\omega)$
follows. This implies that, in particular, in the common case of
uniform tuning with $\bar \lambda = \lambda_1 = \dots = \lambda_p$,
the structural set \emph{only depends on} $X$!
\end{remark}

Coming back to the conjecture whether the structural set always
satisfies $|\Mstruct| \leq \min\{n,p\}$ in case the solutions are
unique, using Theorem~\ref{thm:struct}, we can list the following
simple example with $n=2$ and $p=3$ to show that this cannot be the
case in general. However, note that Theorem~\ref{thm:struct} allows to
compute the structural set and that whenever $|\Mstruct| \leq n$, the
resulting Lasso estimator is, in fact, just equivalent to a
low-dimensional procedure.

\begin{example} \label{ex:fullstructset}
Let 
$$
X = \left(\begin{matrix}
1 & 1 & 0 \\
0 & 1 & 1 
\end{matrix}\right)
$$
and $\lambda = (1,1,1)'$. Then the structural set is clearly given by
$$
\Mstruct = \{1,2,3\},
$$
as $(1,1,0)' \in \col(X') \cap \B_{\{1,2\}}$ and $(0,1,1)' \in
\col(X') \cap \B_{\{2,3\}}$, and $\B_\M \subseteq \B_{\{j\}}$ for any
$j \in \M$ by Corollary~\ref{cor:modelselection}, see
Figure~\ref{fig:shrinkHIGH3} for illustration. Yet the Lasso solutions
for this $X$ and $\lambda$ are always unique which can be checked on
the basis of Theorem~\ref{thm:unique} in the subsequent section.
\end{example}

\begin{figure}[htp]
\centering \includegraphics[width=0.43\textwidth]{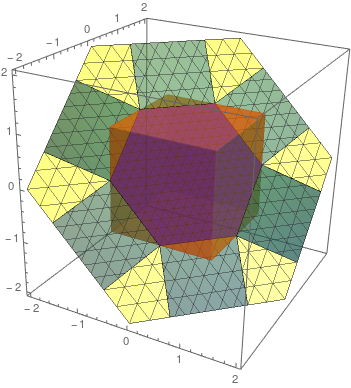}
\caption{\label{fig:shrinkHIGH3} The intersection of $\col(X')$ (in
gray and yellow) with the unit cube (in orange) for $\lambda_1 =
\lambda_2 = \lambda_3 = 1$ and $X = \left(\protect\begin{smallmatrix} 1 & 1 & 0 \\ 
0 & 1 & 1\protect\end{smallmatrix}\right)$, see Example~\ref{ex:fullstructset}. 
The upper left edge is contained in $\B_{\{1,2\}}$, whereas the upper
back edge is contained in $\B_{\{2,3\}}$. (Each $\B_{\{i,j\}}$
contains four parallel edges.) To view this figure in terms of
selection regions, note that the areas corresponding to
single-regressor models are displayed in gray, while the selection
regions that correspond to two-regressor models are displayed in
yellow. The intersection of the $\lambda$-cube with $\col(X')$, which
corresponds to the zero estimator, is displayed in blue.}
\end{figure}

Finally, it is important to note that Theorem~\ref{thm:struct} also
reveals that $\Mstruct$ contains all regressors if the columns of $X$
are scaled to have unit length and the components are tuned uniformly:
For $s = \bar\lambda X'X_j \in \col(X')$, by the Cauchy-Schwarz
inequality, we have $s \in \B_{\{j\}}$ also, leading to $j \in
\Mstruct$. 

\addtocounter{example}{-3}
\begin{example}[continued]
If we rescale the columns of $X$ from Example~\ref{ex:n2p3}, we obtain
$$
\tilde X = \left(\begin{matrix} 
\frac{2}{\sqrt{5}} & 0 & 1\\
\frac{1}{\sqrt{5}} & 1 & 0 
\end{matrix} 
\right).
$$ 
Since the first row is an element of $\B_{\{3\}}$, and the second row
is an element of $\B_{\{2\}}$, clearly $\{2,3\} \subseteq \Mstruct$.
To see that $\{1\} \subseteq \Mstruct$ also, note that the first row
plus $\sqrt{5} - 2$ times the second row lies in $\B_{\{1\}}$ --
yielding a full structural set.
\end{example}
\addtocounter{example}{2}

The above observation may be seen as an argument for rescaling the
regressors before using the Lasso. It may, however, not always be
desirable to so, such as in case the explanatory variables are
observed in the same units, or in the presence of dummy variables.
Also, rescaling the columns may result in changing whether or not the
solutions are unique, an issue addressed in the following section.

\subsection{A Necessary and Sufficient Condition for Uniqueness}
\label{subsec:unique}

We now turn to some results revolving around uniqueness of the Lasso
estimator, which can be obtained with the same geometric approach,
that is, studying the intersection of the $\lambda$-cube with
$\col(X')$. Note that by uniqueness, we mean that for a given $X \in
\R^{n \times p}$, and a given $\lambda \in \R^p$, the Lasso solutions
are unique for all observations $y \in \R^n$.

\cite{Tibshirani13} showed that for a given regressor matrix $X$,
Lasso solutions are unique in the above sense, if the columns of $X$
are \emph{in general position}\footnote{Note that \emph{general
position} does not mean that any selection of $n$ columns of $X$ is
linearly independent, as has sometimes been suggested in the
literature, these two concepts are in fact unrelated.}, which occurs
when no $k$-dimensional affine\footnote{In \cite{Tibshirani13}, the
word ``affine'' is missing, which has caused some confusion in the
literature.} subspace for $k < \min(n,p)$ contains more than $k+1$
elements of the set $\{\pm X_1,\dots,\pm X_p\}$, excluding antipodal
pairs \citep[see p.\ 1463 in][]{Tibshirani13}. In fact, the solutions
are then unique for all choices of the tuning parameter, provided that
all components are tuned equally. As this condition is sufficient, one
may ask whether it is also necessary. The answer to this question is,
in fact, no, as can easily be seen from the example below.

When can non-unique solutions exist? For a given $X \in \R^{n \times
p}$ and $\lambda \in \R^p$, this occurs if and only if there exist $b,
\tilde b \in \R^p$ with $b \neq \tilde b$ and
$$
\col(X') \cap T(b) \cap T(\tilde b) \neq \emptyset.
$$
More concretely, by Theorem~\ref{thm:shrinkHIGH}, and since the Lasso fit
$Xb$ is always unique\footnote{This has been shown in Lemma~1 in
\cite{Tibshirani13} for uniform tuning and can easily be extended to
non-uniform and partial tuning.}, this means that
$$
X'Xb + v = X'X\tilde b + v,
$$ 
where $v \in \col(X') \cap \B_\M$ for some $\M \subseteq
\{1,\dots,p\}$, and $\tilde b_{\M^c} = b_{\M^c} = 0$. Moreover, for $j
\in \M\setminus\M_0$, we have $\sign(b_j) = \sgn(v_j)$ whenever $b_j
\neq 0$, as well as $\sign(\tilde b_j) = \sgn(v_j)$ whenever $\tilde
b_j \neq 0$. Note that we therefore have $Xb = X_\M b_\M = X_\M \tilde
b_\M = X \tilde b$, implying that the columns of $X_\M$ must be
linearly dependent. So non-uniqueness occurs only if $\col(X') \cap
\B_\M \neq \emptyset$ for $\M \subseteq \{1,\dots,p\}$ with linearly
dependent columns in $X_\M$. The following example now immediately
shows that the columns of $X$ being in general position is not
necessary for uniqueness.

\begin{example} \label{ex:nonGP}
Let 
$$
X = \left(\begin{matrix}
1 & 1 & 2 & 0 \\
0 & 0 & 1 & 3
\end{matrix}\right).
$$
Clearly, the columns are not general position, however, all Lasso
solutions are unique for any choice of the tuning parameter, when the
components are tuned uniformly: We have $\col(X') \cap \B_\M =
\emptyset$ whenever $\{1,2\} \subseteq \M$ or $|\M| > 2$. This can
easily be checked using the fact that $v \in \col(X')$ if and only if
$v'w_1 = v'w_2 =0$ for $\ker(X) = \col\{w_1,w_2\}$. Therefore, the
columns of $X_\M$ are linearly independent for any $\M$ that can be
chosen by the Lasso, and all Lasso solutions must be unique.
\end{example}

The example above illustrates the commonly known fact that, if a Lasso
solution is unique, it will contain at most $n$ non-zero entries. We
show that this fact can be sharpened to yield a necessary and
sufficient condition for uniqueness of all Lasso solutions in the
following way: first, we show that if the solution is unique, it in
fact has at most $\rk(X) \leq n$ non-zero components. Second, we prove
that this is not only a necessary, but also a sufficient criterion for
uniqueness.

\begin{theorem}[Uniqueness] \label{thm:unique} 
Let $X \in \R^{n \times p}$ and $\lambda \in \Rgeq^p$. The Lasso
solution is unique for all $y \in \R^n$ if and only if
$$
\col(X') \cap \B_\M = \emptyset 
\text{ for all } \M \subseteq \{1,\dots,p\} \text{ with } |\M| > \rk(X).
$$
\end{theorem}

\begin{proof} ($\implies$) Assume the condition is not
satisfied. Then there exists $v \in \B_\M$ with $|\M| > \rk(X)$ and $v
= X'z$ for some $z \in \R^n$. We show that there is a $y \in \R^n$
such that the corresponding Lasso problem is not uniquely solvable.

If $X_j = 0$ for some $j \in \M_0$, we are done as the corresponding
coefficient may be arbitrary. Note that $X_j = 0$ for $j \in
\M\setminus \M_0$ is not possible: since $v \in \col(X')$, this would
imply $v_j = 0$, but that contradicts $v \in \B_\M$. We therefore
assume that $X_j \neq 0$ for all $j \in \M$.

Since $|\M| > \rk(X)$, there is a column of $X_\M$, say $X_j$ ($X_j
\neq 0$), that can be written as a linear combination of the other
columns. In particular, we can write
$$
dX_j = \sum_{l \in \M\setminus\{j\}} c_l X_l,
$$
where $d = \sgn(v_j)$ if $\lambda_j \neq 0$ and $d = 1$ if $\lambda_j
= 0$. Moreover, let $c = \max_{l \in \M\setminus\{j\}} |c_l| > 0$.
Define $b \in \R^p$ by
$$
b_l  = \begin{cases} 
\frac{d}{2c}  & l = j \\
\sgn(v_l) & l \in \M\setminus\{j\} \\
0 & l \notin \M. 
\end{cases}
$$
Then $b$ is a Lasso solution for $y = z + Xb$ since
$$
X'y = X'Xb + X'z = X'Xb + v \in T(b).
$$
We now construct $\tilde b \in \R^p$, with $\tilde b \neq b$, that is
also a Lasso solution for the same $y$ by
$$
\tilde b_l  = \begin{cases} 
\sgn(v_l) + \frac{c_l}{2c} & l \in \M\setminus\{j\} \\
0 & l = j \text{ or } l \notin \M. 
\end{cases}
$$
Clearly, $b \neq \tilde b$, $\sgn(b_l) = \sgn(\tilde b_l) = \sgn(v_j)$
for $l \in \M\setminus\{j\}$ and
$$
Xb = \sum_{l \in \M} b_l X_l = \sum_{l \in \M\setminus\{j\}} b_l X_l + \frac{d}{2c}X_j
= \sum_{l \in \M\setminus\{j\}} (b_l + \frac{c_l}{2c}) X_l = X\tilde b.
$$
We therefore get
$$
X'y = X'Xb + v = X'X\tilde b + v \in S(\tilde b)
$$
also, implying that both $b$ and $\tilde b$ are Lasso solutions for
the given $y$.

\medskip

\noindent ($\impliedby$) We now prove the other direction. Assume that
there exists $y \in \R^n$ such that non-unique Lasso solutions $b \neq
\tilde b$ exist. As discussed above, this implies the existence of $v
\in \ \col(X') \cap \B_\M$ for some $\M \subseteq \{1,\dots,p\}$ with
$X_\M b_\M = X_\M \tilde b_M$ and $b_{\M^c} = \tilde b_{\M^c} = 0$,
entailing that the columns of $X_\M$ are linearly dependent.

If $|\M| > \rk(X)$, we are done. If $|\M| \leq \rk(X)$, we do the
following. Since we have $\rk(X_\M) < |\M| \leq \rk(X)$, we can pick
$z \in \R^n$ such that $z \in \col(X_\M)^\perp\setminus
\col(X_{\M^c})^\perp$. This is possible since
\begin{align*} 
\col(X_\M)^\perp & \setminus\col(X_{\M^c})^\perp = \emptyset \iff 
\col(X_\M)^\perp \subseteq \col(X_{\M^c})^\perp \\ 
\iff & \col(X_{\M^c}) \subseteq \col(X_\M) \iff \col(X_\M) = \col(X_\M,X_{\M^c}) = \col(X) \\ 
\iff & \rk(X_\M) = \rk(X),
\end{align*}
which is not the case. This $z$ satisfies $(X'z)_\M = (X_\M)' z = 0$
and $(X'z)_{\M^c} = (X_{\M^c})' z \neq 0$, so that we can find $c \in
\R$ such that
$$
\tilde v = v + c\, X'z \in \B_{\tilde\M} \cap \col(X'),
$$
with $\M \subseteq \tilde\M$ and $|\M| < |\tilde\M|$. As long as
$|\tilde\M| \leq \rk(X)$, repeat the steps above with $v = \tilde v$
and $\M = \tilde \M$.
\end{proof}

Note that just as for Theorem~\ref{thm:struct}, the result from the
above theorem depends on $\lambda$ only through the penalization
weights, meaning that for any $\M \subseteq \{1,\dots,p\}$, whenever
$\lambda = \bar \lambda \omega$ for some $\bar \lambda > 0$ and
$\omega \in \Rgeq^p$, we have $\col(X') \cap \B_\M(\lambda) =
\emptyset$ if and only if $\col(X') \cap \B_\M(\omega) \neq \emptyset$
(when indicating the dependence of $\B_\M$ on the tuning parameters).

As mentioned in the preamble of Section~\ref{sec:HIGH},
Theorem~\ref{thm:unique} does not require $p > n$, so that it also covers
the low-dimensional case. Clearly, the condition for uniqueness is
trivially satisfied if $\rk(X) = p$.

\section{Conclusion} 
\label{sec:conclusion}

We give explicit formulae regarding the distribution of the Lasso
estimator in finite-samples, assuming a Gaussian distribution of
$X'\eps$. In the low-dimensional case, we consider the cdf as well as
the density functions conditional on ``active sets'' of the estimator.
Our results exploit the structure of the underlying optimization
problem of the Lasso estimator and do not hinge on the normality
assumption. We also explicitly characterize the correspondence between
the Lasso and the LS estimator: It is shown that the Lasso estimator
essentially creates shrinkage areas around the axes inside which the
probability mass of the LS estimator is compressed into
lower-dimensional densities that can be specified conditional on the
active set of the estimator. As a result, the distribution looks like
a pieced-together combination of Gaussian-like densities. Each active
set has its own distributional piece with dimension depending on the
number of nonzero components, resulting also in point mass at the
origin and mass being distributed along the axes.

The form of the distribution is even more intricate in the
high-dimensional case, in which the estimator may not be unique
anymore. We quantify the relationship between a Lasso solution and the
quantity $X'y$ (rather than the LS estimator as in the low-dimensional
case). We gain valuable insights into the behavior of the estimator by
illustrating that some models may never be selected by the estimator:
The so-called structural set, that contains all covariates that are
part of a Lasso solution for \emph{some} response vector $y$, can be
computed based on a geometric condition involving the regressor matrix
and penalization weights only. In case this structural set has
cardinality less than or equal to $n$, the Lasso is equivalent to a
low-dimensional procedure and results from the $p \leq n$-framework
can be used for inference. We also learn that in case of uniform
tuning and the columns of $X$ scaled to unit length, the structural
set contains all covariates.

Finally, the previous insights allow us to close a gap in the
literature by providing a condition for uniqueness of the Lasso
estimator that is both necessary and sufficient.

\section*{Acknowledgements}
\label{sec:acknowledge}

The authors gratefully acknowledge support from the Deutsche
Forschungsgemeinschaft (DFG) through grant FOR 916 and thank Thomas
Hack and Nikos Dafnis for very insightful discussions on the geometric
aspects of this paper. We also thank two anonymous referees for
their comments.


\end{document}